\documentclass[draft]{tran-l}

\usepackage[T1]{fontenc}
\usepackage[utf8]{inputenc}

\usepackage[british]{babel}
\usepackage{csquotes}
\usepackage{amssymb,amsmath,amsthm,thmtools}
\usepackage{mathtools}	
\usepackage[only,Ydown,mapsfrom]{stmaryrd}
\usepackage{tikz-cd}
	\usetikzlibrary{decorations.pathmorphing}
\usepackage{todonotes}
\usepackage{calrsfs} 
	\let\mathcalcurly\mathcal
	\DeclareMathAlphabet{\mathcal}{OMS}{zplm}{m}{n} 
\usepackage{comment}
\usepackage{xparse}
\usepackage{pbox}
\usepackage[hidelinks,colorlinks,linktoc=all]{hyperref}
\addto\extrasbritish{%
	\def\Snospace~{\S{}}
}

\tikzset{
	symbol/.style={
		draw=none,
		every to/.append style={
			edge node={node [sloped, allow upside down, auto=false]{$#1$}}}
	}
}
\tikzset{
	commutative diagrams/.cd,
	mono/.style={tail}, 
	epi/.style={two heads}, 
	iso/.style={mono, epi, "\sim"{sloped}}, 
	iso'/.style={mono, epi, "\sim"{',sloped}},
	mapsto/.style={|->},
	identity/.style={equal},
	inclusion/.style={"\subseteq"{sloped}},
	linclusion/.style={hook', "\supseteq"{sloped}},
	double central/.style={mono,"\diamond"{description,inner sep=0pt, pos=0.7, sloped}},
	cprod/.style={mono,"\cprod"{description,inner sep=-3pt, pos=0.5, sloped}},
	dashed/.style={densely dashed},
	dotted/.style={dash dot},
	inline/.style={cramped, sep=small},
	cube/.style={sep={3.5em,between origins}},
	row sep/none/.initial=0em,
	align/.style={row sep=none},
	:/.style={r,draw=none, ":"{font=,anchor=center}},
	::/.style={rr,draw=none, ":"{font=,anchor=center}},
	dots/.style={r,draw=none, "\dots"{font=,anchor=center}},
	oplus/.style={draw=none, "\oplus"{font=,anchor=center}},
	times/.style={draw=none, "\times"{font=,anchor=center}},
	label/.style={/tikz/column 1/.append style={
			column sep= 0.75em}},
	smallColSep/.style={/tikz/column #1/.append style={
				column sep = 0.25em}},
	left label/.style={label,/tikz/column 1/.append style={
			anchor=base east,
			"xy"}},
	displaystyle/.style = {
		cells={font=\everymath\expandafter{\the\everymath\displaystyle}},
	},
	longarrow/.style = {shorten <= -3mm, shorten >= -3mm},
	toJordan/.style = {rightsquigarrow,"J"},
	toJordan near end/.style = {rightsquigarrow,"J" near end}
}

\newcommand{\Email}[1]{\href{mailto:#1}{\nolinkurl{#1}}}

\newcommand{\alignbreak}{\\&\hspace{1cm}}

\declaretheorem[style=plain,name=Theorem,numberwithin=section]{thm}

\declaretheorem[style=plain,name=Theorem,sibling=thm]{thmMain}

\declaretheorem[style=plain,name=Proposition,sibling=thm]{prop}

\declaretheorem[style=plain,name=Lemma,sibling=thm]{lem}

\declaretheorem[style=definition,name=Definition,sibling=thm]{defn}
	
\declaretheorem[style=definition,name=Example,sibling=thm]{exmp}
\declaretheorem[style=definition,name=Assumption,sibling=thm]{asm}
\declaretheorem[style=definition,name=Question,sibling=thm]{question}

\declaretheorem[style=remark,name=Remark,sibling=thm]{rem}

\newcommand{\Np}{\mathbb{N}_{+}}
\newcommand{\N}{{\mathbb{N}}}

\newcommand{\Z}{\mathbb{Z}}
\newcommand{\Q}{\mathbb{Q}}
\newcommand{\R}{\mathbb{R}}
\newcommand{\C}{\mathbb{C}}
\newcommand{\F}{\mathbb{F}}
\newcommand{\CP}[1]{\mathbb{CP}^{#1}}
\newcommand{\PP}{\mathbb{P}}
\newcommand{\Affine}{\mathbb{A}}

\DeclareMathOperator{\Grassmann}{Gr}
\DeclareMathOperator{\Stiefel}{V}

\renewcommand{\phi}{\varphi}
\renewcommand{\rho}{\varrho}

\newcommand{\cbox}[2]{\pbox{#1}{\small\relax\ifvmode\centering\fi #2}}

\DeclareMathOperator{\GL}{GL}
\DeclareMathOperator{\SL}{SL}
\DeclareMathOperator{\PSL}{PSL}

\DeclareMathOperator{\SU}{SU}
\DeclareMathOperator{\U}{U}

\DeclareMathOperator{\Homeo}{Homeo}
\DeclareMathOperator{\Diff}{Diff}
\DeclareMathOperator{\Bir}{Bir}
\DeclareMathOperator{\Symp}{Symp}
\DeclareMathOperator{\Ham}{Ham}
\DeclareMathOperator{\Aut}{Aut}
\DeclareMathOperator{\Biholo}{Bih}
\DeclareMathOperator{\Hom}{Hom}

\newcommand{\sphere}[1]{\mathbb{S}^{#1}}
\newcommand{\torus}[1]{{\mathbb{T}^{#1}}}

\newcommand{\complexTorus}[1]{{\mathbb{T}_{#1}}}
\newcommand{\complexCircle}{\mathbb{T}}

\DeclareMathOperator{\ch}{ch}
\DeclareMathOperator{\im}{Im}
\DeclareMathOperator{\rank}{rk}

\newcommand{\id}{\mathrm{id}}

\newcommand{\isom}{\cong}

\newcommand{\acts}{\circlearrowright}
\newcommand{\tensor}{\otimes}
\newcommand{\extTensor}{\operatorname{\boxtimes}}

\newcommand{\cohomology}[3]{H^{#1}(#2;#3)}
\newcommand{\cohomologyRes}[4]{H^{#1}(#2,#3)|_{#4}}

\newcommand{\cprod}{\circ} 
\newcommand{\centralProdPlain}{\mathbin{\cprod}}

\NewDocumentCommand{\centralProj}{o}{\mathrm{p}\IfNoValueTF{#1}{}{_{#1}}}
\NewDocumentCommand{\centralUniv}{o}{\cprod\IfNoValueTF{#1}{}{_{#1}}}

\NewDocumentCommand{\centralProductMax}{m o m}{{#1}\mathbin{\hat\cprod\IfNoValueTF{#2}{}{\!_{#2}}} {#3}}

\NewDocumentCommand{\HeisenbergMono}{o}{\iota\IfNoValueTF{#1}{}{_{#1}}}
\NewDocumentCommand{\HeisenbergEpi}{o}{\pi\IfNoValueTF{#1}{}{_{#1}}}

\NewDocumentCommand{\cbaExtInline}{o m o m o m}{
	\begin{tikzcd}[inline,ampersand replacement=\&]
		\IfNoValueF{#1}{#1:}
		1\ar[r] \& 
		#2\ar[r,central,mono,"\IfNoValueF{#3}{#3}"] \& 
		#4\ar[r,epi,"\IfNoValueF{#5}{#5}"] \& 
		#6\ar[r] \& 
		1
	\end{tikzcd}
}

\DeclareMathOperator{\HH}{H}

\newcommand{\cuptimes}{\mathbin{\smile}}

\newcommand{\embeds}{\rightarrowtail}

\newcommand{\normal}{\lhd}
\DeclareMathOperator{\Center}{Z}
\DeclareMathOperator{\characteristic}{char}

\newcommand{\generate}[1]{{\langle #1\rangle}}

\newcommand{\centralByAbelian}{central-by-abelian}
\newcommand{\CentralByAbelian}{Central-by-abelian}

\newcommand{\HeisenbergFunctor}{\mathop{\mathcal{H}}}

\newcommand{\maximalCBAProj}{\pi_{\mathcal{Z}}}

\DeclareMathOperator{\maxCBAfunctor}{\mathcal{Z}}

\newenvironment{summary}{\begin{quote}\small}{\end{quote}}
\newenvironment{smallpmatrix}{\left(\begin{smallmatrix}}{\end{smallmatrix}\right)}

\newcommand{\divides}{\bigm|}

\newcommand{\ndivides}{%
	\mathrel{\mkern.5mu 
		\ooalign{\hidewidth$\bigm|$\hidewidth\cr$\nmid$\cr}%
	}%
}

\newcommand{\structureSheaf}[1]{\mathcalcurly{O}_{#1}}
\newcommand{\invertibleSheaf}[1]{\mathcalcurly{L}(#1)}

\newcommand{\MumfordH}[1]{H_\mathrm{M}(#1)}
\newcommand{\MumfordTheta}[1]{\mathcalcurly{G}(#1)}
\newcommand{\MumfordSES}[1]{\mathcalcurly{M}(#1)}

\newcommand{\vectorBundleEquivariant}[3]{V_{#1}^{#2}(#3)}

\newcommand{\trivialLineBundle}[1]{\theta_{#1}}
\newcommand{\actionOnTrivialLineBundle}[2]{\Theta_{#1,#2}}

\newcommand{\pD}[1]{p_{#1}}
\newcommand{\LD}[1]{\mathcal{L}_{#1}}

\newcommand{\AutSES}[1]{\Diff_{#1}}
\newcommand{\AutSESGeneral}[1]{\Aut_{#1}}

\newcommand{\from}{\colon}
\newcommand{\leteq }{\coloneqq}
\newcommand{\eqlet}{\eqqcolon}

\newcommand{\arrowInline}[1]{\begin{tikzcd}[inline,ampersand replacement = \&]%
	\!\ar[r,{#1}] \& \!%
\end{tikzcd}}

\makeatletter
\newcommand{\oset}[3][0ex]{%
	\mathrel{\mathop{#3}\limits^{
			\vbox to#1{\kern-2\ex@
				\hbox{$\scriptscriptstyle#2$}\vss}}}}
\makeatother

\newcommand{\toJordan}{\mathrel{\oset[0.1ex]{J}{\rightsquigarrow}}}

\newcommand{\posetFamilies}{\mathfrak{F}}
\newcommand{\JordanLattice}{\mathfrak{J}}

\numberwithin{equation}{section}

\begin{document}

\title[Constructing Highly Symmetric Manifolds and Varieties]{Constructing Highly Symmetric Compact Manifolds and Algebraic Varieties}


\author{Dávid R. Szabó
}
\address{Alfréd Rényi Institute of Mathematics, Reáltanoda u. 13–15., H–1053, Budapest, Hungary
}
\email{szabo.r.david@gmail.com
}
\thanks{{The project
		leading to this application has received funding from the European
		Research Council (ERC) under the European Union’s Horizon 2020
		research and innovation programme (grant agreement No 741420). 
		The author was supported by the
		National Research, Development and Innovation Office (NKFIH) Grant
		K138596.}
}


\date{27th September 2025}

\dedicatory{}

\begin{abstract}
	For every algebraically closed field $k$ and natural number $r$, we construct several algebraic varieties (over $k$) whose birational automorphism group contains every finite nilpotent group of class at most $2$, rank at most $r$ whose order is coprime to the characteristic of $k$. 
This construction is sharp in characteristic $0$, i.e. up to bounded extension, the set of groups from the statement cannot be replaced by a larger one.

Using similar main ideas (with different technical details), for every $r$, we construct several compact manifolds whose diffeomorphism groups contain every finite nilpotent group of class at most $2$, rank at most $r$. 
This result answers a question of Mundet~i~Riera affirmatively and is conjecturally sharp up to bounded extension.  
\end{abstract}

\maketitle

\tableofcontents
\section{Introduction}

The following two notions are essential to describe the finite subgroup structure of certain infinite transformation groups. 

\begin{defn}[{\cite[Definition~2.1]{Popov} motivated by \cite{serre2009Cremona}; generalised by \cite[Definition~8.1]{ProkhorovShramov2014} and \cite[Definition~1]{guld2020finite2nilpotent}}]
\label{defn:Jordan}
	A group $G$ is called 
	\emph{Jordan}, respectively 
	\emph{solvably Jordan}, respectively
	\emph{nilpotently Jordan (of class at most $c$)},   
	if there is an integer $J_{G}$ such that every finite subgroup $F\leq G$ 
	sits in a short exact sequence $1\to N\to F\to B\to  1$ 
	where $|B|\leq J_G$ and $N$ is 
	abelian, respectively 
	solvable, respectively
	nilpotent (of class at most $c$).
\end{defn}
Note that we get an equivalent definition if $N$ is not assumed to be normal in $G$, as replacing a subgroup of index $n$ by its normal core produces a normal subgroup of index at most $n!$.

\begin{defn}\label{defn:rankGroup}
	A group $G$ is \emph{of rank at most $r$} if every subgroup $H\leq G$ can be generated by at most $r$ elements. 
	The infimum of such $r\in\N$ is the \emph{rank} of $G$ (also called the \emph{subgroup rank}).
\end{defn}

\subsection{Algebraic varieties and \autoref{thm:mainBir}}

The structure of finite subgroups of the birational automorphism group of algebraic varieties are known to be controlled by the following statements.

\begin{thm}[{\cite[Theorem 2]{guld2020finite2nilpotent} based on \cite{prokhorov2016jordan}}]\label{thm:birational2Nilpotent}
	The birational automorphism group $\Bir(X)$ of any variety $X$ over a field of characteristic $0$ is nilpotently Jordan of class at most $2$.	
\end{thm}

\begin{thm}[{\cite[Remark 6.9]{ProkhorovShramov2014} or  \cite[Theorem~15]{guld2019finiteDnilpotent} for details}]\label{rem:boundedRankBir}
	For every variety $X$ over a field of characteristic $0$, there exists 
	a constant $R(X)$, only depending on the birational class of $X$, such that the rank of every finite subgroup $F\leq \Bir(X)$ of the birational automorphism group is at most $R(X)$.
\end{thm}

In light of these results, the next statement (which is one of the two main results of this paper) is essentially the strongest possible in characteristic $0$: up to bounded extension, no variety can have a larger set of finite groups acting faithfully via birational automorphisms simultaneously on it, cf. \autoref{rem:poset}. 
\begin{thmMain}[Main statement for varieties]\label{thm:mainBir}
	For every natural number $r$ and every algebraically closed field $k$, 
	there exists an algebraic variety $X_{r,k}$ over $k$ whose 
	birational automorphism group $\Bir(X_{r,k})$ contains an isomorphic copy of 
	every finite nilpotent group $G$ of class at most $2$, rank at most $r$ and order coprime to the characteristic of $k$. 
	In fact, concrete varieties are known to satisfy the conclusion as detailed in \autoref{rem:varietyConstruction}.
\end{thmMain}
\begin{rem}[Constructed varieties]\label{rem:varietyConstruction}
	The proof we give in \autoref{sec:proofA} is constructive and shows that 	\[X_{r,k}\leteq\Big(\prod_{i=1}^r\prod_{j=1}^{\lfloor r/2\rfloor}T_{i,j}\Big) \times Y\] satisfies \autoref{thm:mainBir} 
	where 
	each $T_{i,j}$ is an abelian variety over $k$ admitting a \emph{principal polarisation} (i.e. there exists an ample invertible sheaf on $T_{i,j}$ of degree $1$),
	and $Y$ is a rational variety over $k$ of dimension at least $r$. 
	
	Defining example for such $Y$ is the projective space $\PP_k^r$, 
	while notable examples for such $T_{i,j}$ are:
	\begin{itemize}
		\item an arbitrary \emph{elliptic curve} (e.g. by considering the invertible sheaf corresponding to the Weil divisor of any single (closed) point which is ample by {\cite[Corollary~IV.3.3]{hartshorne}}), or more generally,
		\item an arbitrary \emph{Jacobian variety} by  \cite[Summary~6.11]{Milne1986JacobianVarieties}, and
		\item $(A\times A^\vee)^4$ for an \emph{arbitrary abelian variety} $A$ by  \cite[Remark~16.12]{Milne1986AbelianVarieties} using Zarhin's trick \cite{ZarhinTrick}  (where $A^\vee$ is the dual, cf. \cite[\S III.13]{mumford2008abelian}).
	\end{itemize}
	More generally, we will show in \autoref{rem:mainTheoremAdmissible} that $\prod_{j=1}^{\lfloor r/2\rfloor}T_{i,j}$ could actually be replaced by an arbitrary $( r/2)$-admissible abelian variety, see \autoref{defn:s-admissible}.
	
	Posed by an anonymous referee, it would be interesting to see for which abelian varieties $A$ and rational varieties $Y$  does the conclusion of  \autoref{thm:mainBir} hold for $X=A\times Y$. 
	Note that $\Bir(X)$ of any such $X$ is known to be non-Jordan \cite[Corollary~1.4]{zarhin2014}.
	
\end{rem}

\begin{rem}[Brief history in characteristics $0$]
\label{rem:historyBirational}
A reformulation of the classical result of  Camille Jordan from 1877 states that 
$\GL_n(\C)$ is Jordan for every $n\in \Np$ \cite{Jordan1877}. 
In 2009, Jean-Pierre Serre showed that the birational automorphism group of the projective plane (the so-called Cremona group) over an arbitrary field of characteristic $0$ is Jordan \cite[Theorem~5.3]{serre2009Cremona}. 
These results motivated Vladimir L. Popov in 2011 to introduce the notion of the Jordan property as in \autoref{defn:Jordan}. 
He proved that the birational automorphism group of every surface is Jordan except possibly for the product of an elliptic curve and the projective line \cite[\S2.2]{Popov}. 
Using David Mumford's theta groups \cite{mumford1966equations}, in 2014 Yuri G. Zarhin showed the failure of the Jordan property for this missing $2$-dimensional variety 
by exhibiting an infinite list of class $2$ nilpotent $p$-groups inside the birational automorphism group \cite[Theorem~1.2]{zarhin2014}. 

There are two natural directions for the development of the theory from this point. 
The first one is to classify varieties whose birational automorphism group is Jordan. (We consider the birational automorphism group as it is known that the \emph{automorphism} group of every projective variety over an algebraically closed field of characteristic $0$ is Jordan by \cite[Theorem~1.6]{MengZhang2018}.)
While the second one is to find a relaxation of the Jordan property so that the birational automorphism group of \emph{every} variety has this weaker property. 
Yuri Prokhorov and Constantin Shramov made progress in both directions in 2014-2016
using the Minimal Model Program and assuming the Borisov-Alexeev-Borisov conjecture (which has been proved since then by Caucher Birkar \cite[Theorem~1.1]{birkar2016singularities}). 
First, they proved that over a field of characteristic $0$, 
the birational automorphism group of non-uniruled varieties  \cite[Theorem~1.8(ii)]{ProkhorovShramov2014} and of rationally connected varieties \cite[Theorem~1.8]{prokhorov2016jordan} are both Jordan. 
Second, they proved  
the birational automorphism group of every variety is solvably Jordan \cite[Proposition 8.6]{ProkhorovShramov2014}. 
Attila Guld (\autoref{thm:birational2Nilpotent}) strengthened this by replacing the `solvably Jordan' property with the stronger `nilpotently Jordan of class at most $2$' property. 
In light of \autoref{thm:mainBir}, no further strengthening  is possible in this second direction in the sense of \autoref{rem:poset}. 

The proof of \autoref{thm:mainBir} from \autoref{sec:birational} builds strongly on the ideas of the aforementioned  \cite{zarhin2014} as well as on the description of Mumford's theta groups \cite{mumford1966equations} and that of nilpotent groups of class at most $2$ by the author \cite{Heisenberg}.
\end{rem}

We can express the sharpness of \autoref{thm:mainBir} in characteristic $0$ more formally as follows in the language of lattices.
\begin{rem}[Least upper bound in the lattice of bounded extensions]\label{rem:poset}
	Let $\posetFamilies$ consist of the sets of (isomorphism classes of) finite groups which are closed under taking normal subgroups. 
	Define a preorder $\toJordan$ on $\posetFamilies$ by letting  $\mathcal{A}\toJordan \mathcal{B}$ if for every $r\in\N$ 
	there exists an integer $n$ such that every $[B]\in\mathcal{B}$ of rank at most $r$ sits in a short exact sequence $1\to A\to B\to Q\to  1$ of groups 
	for some $[A]\in\mathcal{A}$ and $|Q|\leq n$. 
	The main object in this remark is the the \emph{lattice of bounded extensions} which is defined to be the (bounded distributive) lattice $(\JordanLattice,\preceq)$  induced by the dual of $(\posetFamilies,\toJordan)$. 
	More concretely, $\JordanLattice=\posetFamilies/\sim$ is the set of $\sim$\nobreakdash-equivalence classes of $\posetFamilies$ where for any $\mathcal{A}, \mathcal{B}\in \posetFamilies$, we define $\mathcal{A}\sim \mathcal{B}$ 
	if and only if $\mathcal{A}\toJordan \mathcal{B}$ and $\mathcal{B}\toJordan \mathcal{A}$. 
	Write $[\mathcal{A}]\in \JordanLattice$ for the $\sim$\nobreakdash-equivalence class of $\mathcal{A}\in\posetFamilies$. 
	The partial order $\preceq$ on $\JordanLattice$ is then defined by $[\mathcal{A}]\preceq [\mathcal{B}]$ if and only if 
	$\mathcal{B}\toJordan \mathcal{A}$.
	
	In this lattice, define elements $\mathcal{A}\preceq\mathcal{N}_2\preceq \mathcal{S}\in\JordanLattice$ and $\mathcal{F}(G)\in\JordanLattice$ respectively to be the $\sim$\nobreakdash-equivalence class of the set of (isomorphism classes of) finite abelian groups;
	respectively finite nilpotent groups of class at most $2$;
	respectively finite solvable groups;
	respectively finite subgroups of a group $G$. 

	Note that $\mathcal{F}(\Aut(X))\preceq \mathcal{A}$ for every variety $X$ over a fixed field $k$ of characteristic $0$ by \cite[Theorem~1.6]{MengZhang2018}. 
	In fact, it is the least upper bound in $\JordanLattice$, i.e. \[\sup_X \mathcal{F}(\Aut(X)) =\mathcal{A},\] 
	as every finite abelian group on $n$ generators appears as the subgroup of $\Aut(\PP_k^n)$. 
	
	Similarly for birational automorphism groups, we have $\mathcal{F}(\Bir(X))\preceq \mathcal{S}$ for every $X$ by \cite[Proposition 8.6]{ProkhorovShramov2014}. 
	A stronger upper bound of $\mathcal{F}(\Bir(X))\preceq \mathcal{N}_2$ is given by 
	\cite[Theorem 2]{guld2020finite2nilpotent}. 
	In fact, the sharpness of \autoref{thm:mainBir} is equivalent to 
	this upper bound being the \emph{least} one, i.e. 
	\[\sup_X \mathcal{F}(\Bir(X)) =\mathcal{N}_2.\]
	\autoref{thm:mainBir} also gives a concrete sequence $\mathcal{F}(\Bir(X_{1,k}))$, $\mathcal{F}(\Bir(X_{2,k}))$, \dots in $\JordanLattice$ realising this least upper bound $\mathcal{N}_2$.
\end{rem}

The case of positive characteristic $p>0$ is more complicated and is far less understood.
For example, not even $\GL_n(\overline{\F}_p)$ is Jordan because it has finite subgroups isomorphic to $F=\SL_n(\F_{p^r})$ for every $r\geq 3$ (and $\PSL_n(\F_{p^r})=F/\Center(F)$ is simple). 
Michael Larsen and Richard Pink generalised Jordan's theorem to $\GL_n(k)$ in case of arbitrary fields $k$, see \cite[Theorem~0.2]{LarsenPink}. 
Note that \autoref{thm:mainBir} shows that there are varieties over every algebraically closed field $k$ whose birational automorphism group is not Jordan. 
The previous examples show that in characteristic $p>0$, it may help to disregard finite subgroups whose order is a multiple of $p$.

\begin{defn}[{cf. \cite[Definition 1.14]{ChenShramov2022automorphisms} and \cite[Definition 1.2]{Hu_2020}}]
	For a prime number $p$, call a group $G$ \emph{nilpotently generalised $p$-Jordan of class at most $c$} 
	if there is an integer $J_{G,p}$ such that every finite subgroup $F\leq G$ whose order is not divisible by $p$ 
	sits in a short exact sequence $1\to N\to F\to B\to  1$ 
	where $N$ is nilpotent of class at most $c$ and $|B|\leq J_{G,p}$. 
\end{defn}

Similarly to the characteristic $0$ case, the birational automorphism group of every low dimensional variety satisfies this property. 
\begin{thm}[Special case of {\cite[Theorem 1.15]{ChenShramov2022automorphisms}}]
	If $k$ is a field of characteristic $p > 0$ and $S$ is a geometrically
	irreducible algebraic surface over $k$, then $\Bir(S)$ is nilpotently generalized $p$-Jordan group of class at most $2$.
\end{thm}

\begin{rem}
	It would be interesting to see if this statement holds in higher dimensions, 
	see  \cite[\S1]{Hu_2020} and \cite[\S14]{ChenShramov2022automorphisms}.
	Note that \autoref{thm:mainBir} shows the necessity of nilpotent groups of class at most $2$ in any such potential statement.
\end{rem}

See \cite{guld2020finite2nilpotent}, \cite{phd}, \cite{Hu_2020}, \cite{ChenShramov2022automorphisms}  and their references for more background.

\subsection{Compact manifolds and \autoref{thm:mainDiff}}
The situation for the diffeomorphism (or even homeomorphism) group of compact manifolds is surprisingly similar. 
Recently Balázs Csikós, László Pyber  and Endre Szabó proved the following more general form of the revised conjecture of Étienne Ghys \cite{GhysTalk2015}, cf. \autoref{thm:birational2Nilpotent}

\begin{thm}[{\cite[Theorem 1.3]{nilpotentJordanHomeo}}]\label{thm:homeoNilpotent}
	The homeomorphism group $\Homeo(M)$ of every compact topological manifold $M$ is nilpotently Jordan.
\end{thm}

The classical result \cite{Mann1963} implies the following bound on the rank, cf. \autoref{rem:boundedRankBir}.

\begin{thm}[{\cite[Theorem~1.8]{CsikosMundetPyberSzabo}}]\label{lem:boundedRankDiff}
	For every compact manifold $M$, there exists a constant $R(M)$ such that the rank of every finite subgroup $F\leq \Homeo(M)$ of the homeomorphism group is at most $R(M)$.
\end{thm}

We prove the following analogue of \autoref{thm:mainBir} for compact manifolds, which is the second main result of this paper.
\begin{thmMain}[Main statement for manifolds]\label{thm:mainDiff}
	For every natural number $r$, 
	there exists smooth connected compact manifold $M_r$ such that 
	its diffeomorphism group $\Diff(M_r)$ contains an isomorphic copy of 
	every finite nilpotent group $G$ of class at most $2$ and rank at most $r$. 
	In fact, concrete manifolds  are known to satisfy the conclusion as detailed in \autoref{rem:manifoldConstruction}.
\end{thmMain}

\begin{rem}[Constructed manifolds]\label{rem:manifoldConstruction}
	The proof given in \autoref{sec:proofB} is constructive and shows that \[M_r\leteq \torus{2r\lfloor r/2\rfloor}\times \prod_{i=1}^r Y_i\] satisfies the statement 
	where $\torus{n}$ is the $n$-torus, and $Y_i$ can be any of 
	\begin{itemize}
		\item the \emph{sphere} $\sphere{2t-1}$, 
		\item the \emph{(special) unitary groups} $\SU(t)$, $\U(t)$, or more generally, 
		\item any \emph{Stiefel manifold} $\Stiefel_k(\C^t)\isom \U(t)/\U(t-k)$ for $1\leq k\leq t$,  
		\item the \emph{complex projective space} $\CP{t}$, or more generally 
		\item any \emph{Grassmann manifold} $\Grassmann_k(\C^{t+1})$ for $1\leq k\leq t+1$,
	\end{itemize}
	for every $t\geq t_0(r)\leteq R_3(2\lfloor r/2\rfloor)\in\N$ from \autoref{prop:complementActionExists}.
	Note that $t_0(1)=3$, $t_0(2)=4$, $t_0(3)=4$, $t_0(4)=21$ by \autoref{rem:R1}, 
	\autoref{rem:R2} and \autoref{rem:R3}. 
	In general, we have $t_0(r)\leq 2^{r+3} \left(\left\lfloor\frac{r}{2}\right\rfloor+1\right)!$ by \autoref{rem:R3}.
\end{rem}

Since the rank of every group of order $p^r$ is at most $r$, \autoref{thm:mainDiff} answers the following question affirmatively.
\begin{question}[Ignasi Mundet~i~Riera, correspondence from 2018]
	Is it possible to construct for every $r$ a compact manifold supporting effective actions, for every prime $p$, of every group of cardinal $p^r$ and nilpotency class $2$?
\end{question}

\begin{rem}[Brief history]
	For manifolds, one of the main driving forces of this topic was the following \emph{conjecture of Ghys} originally from the 1990s asserting that $\Diff(M)$ is Jordan for every compact smooth manifold $M$ \cite[\S13.1]{FisherSurvey2008-nonARXIV}. 
	(The compactness assumption is important 
	as there exists a non-compact $4$-manifold whose diffeomorphism group contains an isomorphic copy of every finitely presented group \cite[Corollary~1]{Popov2013}, cf. \cite[Theorem~6]{popov2018jordan} for an analogous statement for complex manifolds.)
	This conjecture was verified affirmatively in many cases by Bruno P. Zimmermann  \cite[Theorem~1]{Zimmermann2014} and  Mundet~i~Riera\ \cite[Theorem~1.4]{Riera2010Torus},  \cite[Theorem~1.2]{Riera2018}. 
	However, being motivated by \cite{zarhin2014} mentioned in \autoref{rem:historyBirational}, Csikós, Pyber and Szabó found the first counterexamples where various rank $2$ Heisenberg groups act simultaneously on a sphere-bundle over the $2$-torus \cite{CsPSz}. 
	(See \cite[\S5.2.6]{phd} by the author for an explicit example of such actions.) 
	This was extended by Mundet~i~Riera to certain higher rank Heisenberg $p$-groups acting on fibre bundles over higher dimensional tori \cite{Riera}. 
	The methods of this paper were extended further to every special $p$-group of bounded rank by the author \cite{specialGroups}. 
	\autoref{thm:mainDiff} above generalises all of these results.

	We remark that similar actions of Heisenberg groups on fibre bundles over the $2$-torus  appeared in \cite{MundetT2S2} when analysing the Jordan property of a symplectomorphism group. 
\end{rem}

Being motivated by \autoref{thm:birational2Nilpotent}, 
Csikós, Pyber and Szabó asked the following question about bounding the nilpotency class from \autoref{thm:homeoNilpotent}.
\begin{question}[{\cite[Question 1.5]{nilpotentJordanHomeo}}]
	Is it true that the diffeomorphism group of every compact manifold is nilpotently Jordan \emph{of class at most $2$}?	
\end{question}
Considering \autoref{thm:mainDiff}, this is the strongest possible structure to ask for. In the case of a positive answer, \autoref{thm:mainDiff} would be sharp similarly to \autoref{thm:mainBir}. 
Mundet~i~Riera and Sáez--Calvo gave an affirmative answer in low dimensions. 
\begin{thm}[{\cite[Theorem 1.1]{Riera4dim}}]
	The diffeomorphism group $\Diff(M)$ of every compact manifold $M$ of dimension at most $4$ without boundary is nilpotently Jordan of class at most $2$.
\end{thm}

Similar results are known in related categories. 
If $(X,\omega)$ is a compact connected symplectic manifold, then the group $\Ham(X,\omega)$ of Hamiltonian diffeomorphisms is Jordan \cite[Theorem~1.1]{RieraHamSymp}.
If furthermore the first Betti number of $X$ is $0$, then the group $\Symp(X,\omega)$ of symplectomorphisms is also Jordan \cite[Theorem~1.3]{RieraHamSymp}.
If $(X,\omega)$ is a compact symplectic $4$-manifold, then $\Symp(X,\omega)$  is Jordan.
\cite[Theorem 1.7]{Riera4dim}.
If $X$ is a smooth compact $4$-manifold with an almost complex structure $J$, then the group $\Aut(X,J)$ of diffeomorphisms preserving $J$ is Jordan \cite[Theorem~1.6]{Riera4dim}. 
As noted above, the automorphism group of a non-compact complex manifold may be nowhere near being Jordan, but that of a compact $2$-dimensional complex manifold is  \cite[\S4]{popov2018jordan}. 
For further details, the interested reader may read the survey of Mundet i Riera on the development of the topic \cite{Mundet2023Survey}.

\subsection{Structure of the paper and overview of ideas} 
The paper is organised as follows. 
In \autoref{sec:groupTheory}, we collect some group theoretic results from \cite{Heisenberg} by the author that are needed for the proof of both of the main statements of this paper (\autoref{thm:mainBir}, \autoref{thm:mainDiff}). The key result here is that it is enough to consider finite Heisenberg groups with cyclic centre (of bounded rank) instead of general nilpotent groups of class at most $2$ (of bounded rank).

In \autoref{sec:birational},  we prove \autoref{thm:mainBir} (from \autopageref{thm:mainBir}), and in \autoref{sec:diff}, we prove \autoref{thm:mainDiff} (from \autopageref{thm:mainDiff}). 
While the technical details in these sections are sometimes quite different, the main flows of the two proofs are surprisingly analogous. Both proofs are presented in a self-contained way and can be read independently of each other.

In the first step, we parametrise (the commutator map on) the Heisenberg group 
by a tuple of natural numbers using \cite{mumford1966equations} (\autoref{sec:algebraic-parametrisation}),
and by a discrete Hermitian form (\autoref{sec:parametrisationHermitianForm}). In the analogy, the entries of the tuple correspond to the invariant factors of the abelian group on which the Hermitian form is defined.

In the second step, for every potential values of the previous parameters, we assign a space: first, an ample invertible sheaf on certain abelian varieties, most notably on product of elliptic curves (in \autoref{sec:algebraic-associatedGroup} using Mumford's theta groups \cite{mumford1966equations});
second a line bundle over the complex torus (in \autoref{sec:associatedLineBundle} using the Appell--Humbert theorem \cite{ComplexAbelianVarieties}). 
The dimensions of these base spaces are bounded by the rank of the Heisenberg group.
For the analogy, note that elliptic curves defined over $\C$ are the complex tori that embed into the complex projective plane. 

In the third step, we exhibit a faithful action of the Heisenberg group on the spaces from the previous step. 
We mainly include \autoref{sec:algebraic-associatedAction} to display the origins of \autoref{sec:diffActionOnLineBundle} where we introduce a general notion of the action of a short exact sequence of groups on a fibre bundle. 
These actions split into two parts: an action on the base space and an action on the fibres. 
Note that both constructions essentially depend on the Heisenberg group we started from.

In the fourth step, we eliminate the dependency of the bundles and the action on the Heisenberg group from the previous step. 
We do so by constructing a new action on a fixed (`uniform') space using the existing action.  
In fact, we show that all of the spaces previously constructed can be embedded into the trivial vector bundle of suitable (bounded) rank such that all of these actions can be extended to this trivial bundle. 
In \autoref{sec:birationalUniformisation} this is done by noticing that every line bundle is birational to the trivial line bundle. 
The situation is much more technical for manifolds in \autoref{sec:uniformisation}. The idea of the method used here originates from \cite{Riera} and uses K-theory, and computations in the cohomology ring leading to a modular number theory problem of Waring type. 

For \autoref{thm:mainDiff} about manifolds, we need an extra step of replacing the uniform \emph{non-compact} space with a compact one (\autoref{sec:compactification}). 
To do so, we use the standard topological constructions of Stiefel and Grassmann bundles. 
This can be considered as an analogue of replacing the quasi-projective variety of \autoref{sec:birationalUniformisation} with a projective variety.

Finally, we prove \autoref{thm:mainBir} about varieties in \autoref{sec:proofA}, and \autoref{thm:mainDiff} about manifolds in \autoref{sec:proofB} by putting the pieces together developed during the previous steps as indicated by the next map.
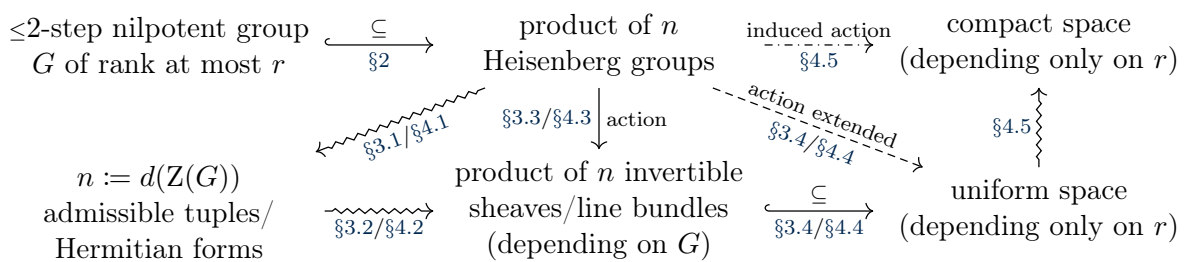
\[
\begin{tikzcd}[column sep=3.5em,row sep=2em]
	\cbox{0.23\textwidth}{nilpotent group $G$\\ of class at most $2$\\ and rank at most  $r$} 
	\ar[r,mono,"\autoref{sec:groupTheory}"']
	& \cbox{0.25\textwidth}{product of $n\leteq d(\Center(G))$ Heisenberg groups} 
	\ar[dl,rightsquigarrow,"\autoref{sec:algebraic-parametrisation}/\autoref{sec:parametrisationHermitianForm}"' sloped]
	\ar[d,"\text{action}","\autoref{sec:algebraic-associatedAction}/\autoref{sec:diffActionOnLineBundle}"']
	\ar[dr,dashed,"\text{action extended}" sloped,"\autoref{sec:birationalUniformisation}/\autoref{sec:uniformisation}"' sloped]
	\ar[r,dotted,"\text{induced action}","\autoref{sec:compactification}"']
	& \cbox{0.25\textwidth}{compact space (depending only on $r$)}
	\\
	\cbox{0.23\textwidth}{$n$ admissible tuples/\\Hermitian forms}
	\ar[r,rightsquigarrow,"\autoref{sec:algebraic-associatedGroup}/\autoref{sec:associatedLineBundle}"']
	& \cbox{0.25\textwidth}{product of $n$ invertible sheaves/\\line bundles  (depending on $G$)} \ar[r,inclusion,"\autoref{sec:birationalUniformisation}/\autoref{sec:uniformisation}"']
	& \cbox{0.25\textwidth}{uniform space (depending only on $r$)} 
	\ar[u,rightsquigarrow,"\autoref{sec:compactification}"]
\end{tikzcd}\]

The present paper is the second of two (starting with \cite{Heisenberg}) compiled from the thesis of the author \cite{phd}.

\subsection{Notation}
$|X|$ denotes the \emph{cardinality} of a set $X$. 
$\Np\subset \Z$ is the set of positive integers, $\N=\{0\}\cup \Np$. 
%
%
We apply functions \emph{from the left}. 
%
Every diagram is implicitly assumed to be \emph{commutative} unless otherwise stated. 
The arrow $\arrowInline{mono}$ indicates a \emph{monomorphism} (or an injective map), 
$\arrowInline{epi}$ means an \emph{epimorphism} (or a surjective map), 
and these arrow notations can be combined. 
We write $\arrowInline{identity}$ for the \emph{identity map}, 
and $\arrowInline{inclusion}$ for the set theoretical inclusion map.
In bigger diagrams, we use $\arrowInline{solid}$, 
$\arrowInline{dashed}$ or 
$\arrowInline{dotted}$ to indicate the `chronological order': the further it is from a solid one, the later it appears in the construction.

Let $G$ denote a group. 
We denote the \emph{identity element} of $G$ by $1$, or sometimes by $0$ when $G$ is an additive abelian group. 
By abuse of notation, we also write $1$ or $0$ for the \emph{trivial group}. 
For a subset $S\subseteq G$, $\generate{S}$ denotes the subgroup generated by $S$, 
and write $\generate{g_1,\dots,g_n}\leteq \generate{\{g_1,\dots,g_n\}}$.
%
%
$N\normal G$ means that $N$ is a \emph{normal subgroup} of $G$.
%
The \emph{commutator subgroup} (or \emph{derived subgroup}) is denoted by $G'\leteq [G,G]$.
$\Center(G)$ is the \emph{centre} of $G$.
%
%
We denote by 
$\exp(G)\leteq \inf\{n\in \Np:\forall g\in G\, g^n = 1\}$ the \emph{exponent} of a group $G$. (By abuse of notation, $\exp\from \C\to \C$ sometimes denotes the exponential function.)
Write $d(G)$ for the \emph{cardinality of the smallest generating set}. 

For a topological space $X$, we denote by $\trivialLineBundle{X}\from X\times \C\to X$ the trivial complex line bundle over $X$.

We denote the characteristic of a field $k$ by $\characteristic(k)$.

\section{Group theory}\label{sec:groupTheory}

\begin{summary}
	We collect the necessary notions and statements from \cite{Heisenberg} by the author that form the common group theoretical core for the proof of both \autoref{thm:mainBir} and \autoref{thm:mainDiff}. 
	We realise every finite nilpotent group of class at most $2$ having bounded rank as a certain subgroup of product of suitable Heisenberg groups of controlled invariants.
\end{summary}

\begin{defn}[{\cite[Definition 3.5]{Heisenberg}}]\label{defn:centralByAbelian}
	A short exact sequence 
	$\epsilon:1\to C\xrightarrow{\iota} G\xrightarrow{\pi} M\to 1$
	of groups is called a \emph{\centralByAbelian{} extension}, if $\iota(C)\subseteq\Center(G)$ and $M$ is abelian. 
	This extension $\epsilon$ is \emph{non-degenerate}\index{\centralByAbelian{} extension!non-degenerate} if $\iota(C)=Z(G)$.
\end{defn}

\begin{defn}[Heisenberg group, {\cite[Definition 4.1]{Heisenberg}}]\label{def:H}
	Let $A$, $B$ and $C$ be $\Z$-modules and 
	$\mu\from A\times B\to C$ a $\Z$-bilinear map. 
	We call $\mu$ \emph{non-degenerate} if $\mu(a,B)=0$ implies $a=0$ and $\mu(A,b)=0$ implies $b=0$. 
	Define the associated \emph{Heisenberg group} as
	$\HH(\mu)\leteq A\ltimes_\phi(B\times C)$
	where $\phi\from A\to \Aut(B\times C),a\mapsto((b,c)\mapsto (b,\mu(a,b)+c)$. 
	Call $\HH(\mu)$ \emph{non-degenerate} if $\Center(\HH(\mu))=\{(0,0,c):c\in C\}$.
	Define a \centralByAbelian{} extension 
	\[\HeisenbergFunctor(\mu):1\to C\xrightarrow{\HeisenbergMono[\mu]} \HH(\mu) \xrightarrow{\HeisenbergEpi[\mu]} A\times B \to 1\]
	by 
	$\HeisenbergMono\leteq \HeisenbergMono[\mu]:c\mapsto(0,0,c)$ and $\HeisenbergEpi\leteq \HeisenbergEpi[\mu]:(a,b,c)\mapsto (a,b)$. 
\end{defn}

\begin{rem}\label{rem:Hdef}
	More explicitly, the group structure on $\HH(\mu)$ is given by  
	\begin{equation}\label{eq:Hdef}
		(a,b,c)*(a',b',c')  = (a+a', b+b',c + \mu(a,b')+c')
	\end{equation} with 
	$(0,0,0)$ being the identity and 
	$(a,b,c)^{-1} =(-a,-b,\mu(a,b)-c)$ the inverse. 
	The notion of non-degeneracy for $\mu$, $\HH(\mu)$ and $\HeisenbergFunctor(\mu)$ all coincide.	

Informally, we may identify  $(a,b,c)\in \HH(\mu)$ by the $3\times 3$ upper triangular `matrix'  $\begin{smallpmatrix}1&a&c\\0&1&b\\0&0&1\end{smallpmatrix}$, 
as then \eqref{eq:Hdef} takes the form 
$\begin{smallpmatrix}1&a&c\\0&1&b\\0&0&1\end{smallpmatrix}*\begin{smallpmatrix}1&a'&c'\\0&1&b'\\0&0&1\end{smallpmatrix}\leteq \begin{smallpmatrix}1&a'+a&c'+\mu(a,b')+c\\0&1&b'+b\\0&0&1\end{smallpmatrix}$ 
resembling the usual matrix multiplication (twisted by $\mu$).
\end{rem}

\begin{lem}[Functoriality]\label{lem:Hfunctorial}
		The commutative diagram on the left involving $\Z$-bilinear rows induces a morphism 
		$\HeisenbergFunctor(\mu_1)\to \HeisenbergFunctor(\mu_2)$ of short exact sequences, i.e. a commutative diagram on the right.
		\[\begin{tikzcd}
		A_1\times B_1 \ar[r,"\mu_1"] \ar[d,"\lambda"] & C_1 \ar[d,"\kappa"] \\
		A_2\times B_2 \ar[r,"\mu_2"] & C_2
	\end{tikzcd}\qquad
\begin{tikzcd}
			1\ar[r] 
			& C_1 \ar[r, mono,"{\HeisenbergMono[\mu_1]}"]\ar[d,"\kappa"] 
			& \HH(\mu_1) \ar[r,"{\HeisenbergEpi[\mu_1]}"] \ar[d,dashed,"\exists\gamma"] 
			& A_1\times B_1 \ar[r] \ar[d,"\lambda"]
			& 1 
			\\
			1\ar[r] 
			& C_2 \ar[r, mono,"{\HeisenbergMono[\mu_2]}"] 
			& \HH(\mu_2) \ar[r,epi,"{\HeisenbergEpi[\mu_2]}"] 
			& A_2\times B_2\ar{r} 
			& 1
		\end{tikzcd}\]
\end{lem}
\begin{proof}
		One can check that $\gamma:(a,b,c)\mapsto (\lambda(a),\lambda(b),\kappa(c))$ is a group morphism satisfying the statement.
\end{proof}

The next statement is essential for the current paper.
\begin{lem}[{\cite[Theorem 5.12]{Heisenberg}}]\label{thm:embedToH}
	For every finite nilpotent group $G$ of class at most $2$, there exist  
	finite $\Z$-modules $A_i$, $C_i$, 
	and non-degenerate $\Z$-bilinear maps $\mu_i\from A_i\times A_i\to C_i$ for $1\leq i\leq d(\Center(G))$ and group monomorphisms $\zeta, \delta, \nu$ fitting the commutative diagram
	\begin{equation}\label{diag:embeddingToHeisenberg}
		\begin{tikzcd}[label,displaystyle]
			\maxCBAfunctor(G)\ar[:] \ar[d,"\exists f",mono,dashed]& 
			1\ar[r] & 
			\Center(G) \ar[r,inclusion]\ar[d,mono,"\exists \zeta",dashed] & 
			G \ar[r,epi,"\maximalCBAProj"]\ar[d, mono,"\exists \delta",dashed] & 
			G/\Center(G) \ar{r}\ar[d,mono,"\exists \nu",dashed] & 
			1 
			\\
			\prod_{i=1}^{\mathclap{d(\Center(G))}} \HeisenbergFunctor(\mu_i)\ar[:] &
			1\ar[r] & 
			\prod_{i=1}^{\mathclap{d(\Center(G))}} C_i \ar[r,inclusion] & 
			\prod_{i=1}^{\mathclap{d(\Center(G))}} \HH(\mu_i) \ar[r,"{\prod_i \HeisenbergEpi[\mu_i]}"] & 
			\prod_{i=1}^{\mathclap{d(\Center(G))}} A_i\times A_i \ar[r]   & 
			1
		\end{tikzcd}
	\end{equation}
	such that $C_i=\Center(\HH(\mu_i))$ are all cyclic, $1\leq d(A_i)\leq \frac{1}{2}d(G)$ and  
	any prime divisor of the order of any group above also divides $|G|$.
\end{lem}

\section{Algebraic varieties and \autoref{thm:mainBir}}\label{sec:birational}
\begin{summary}
	In this section, we prove the first main statement of the paper: \autoref{thm:mainBir} from \autopageref{thm:mainBir}.
\end{summary}
	\subsection{\CentralByAbelian{} extensions via admissible tuples}
\label{sec:algebraic-parametrisation}

\begin{summary}
	Following \cite{mumford1966equations}, we associate an abstract theta group to every admissible tuple. 
	Then we show that every non-degenerate Heisenberg group with cyclic centre embeds to one of these abstract theta groups.
\end{summary}

\begin{defn}[cf. {\cite[p.294]{mumford1966equations}}]\label{defn:MumfordDeltaGroups}
	Call a $t$-tuple  $\delta\leteq (d_1,\dots,d_t)$ of positive integers \emph{admissible for a field $k$} if 
	$1<d_t\divides d_{t-1}\divides \dots \divides d_1$ and $\characteristic(k) \ndivides d_1$. 
	For every admissible tuple, we assign a non-degenerate \centralByAbelian{} extension 
	\[\begin{tikzcd}[left label]
		\MumfordSES{\delta}\ar[:] &
		1\ar[r] & 
		k^\times \ar[r,mono,"\iota_\delta"]  & 
		\MumfordTheta{\delta}\ar[r,epi,"\pi_\delta"]   & 
		\MumfordH{\delta} \ar[r] & 
		1
	\end{tikzcd}\]
	where 
	\begin{itemize}
		\item $\MumfordH{\delta}\leteq K(\delta)\times \Hom(K(\delta),k^\times)$ for $K(\delta)\leteq \prod_{i=1}^t \Z/d_i\Z$, 
		\item the \emph{abstract theta group} $\MumfordTheta{\delta}$ is defined on the set $k^\times\times \MumfordH{\delta}$ by the group operation $(c,b,\alpha)\cdot (c',b',\alpha')\leteq (cc'\alpha'(b),b+b',\alpha\alpha')$, and 
		\item the morphisms are given by $\iota_\delta\from c\mapsto (c,0,0)$ and $\pi_\delta\from (c,b,\alpha)\mapsto (b,\alpha)$.
	\end{itemize}
	
\end{defn}

Abstract theta groups and Heisenberg groups are closely related. In fact, the next statement was the initial motivation for \autoref{def:H}. 
\begin{prop}[Parametrisation]\label{lem:HerisenbergToMumfordDelta}
	Let $\HH(\mu\from A\times A\to C)$ be a finite non-de\-ge\-ne\-rate Heisenberg group with cyclic centre. 
	Let $k$ be an algebraically closed field such that $\characteristic(k)\ndivides |\HH(\mu)|$. 
	Then there exists a (unique) admissible $d(A)$-tuple $\delta(\mu)$ for $k$ and vertical morphisms making the following diagram commute. 
	\begin{equation}\label{diag:HeisenbergEmbedsTOMumford}
		\begin{tikzcd}[left label]
			\HeisenbergFunctor(\mu) \ar[:]& 
			1\ar[r] & 
			C\ar[r,mono,"{\HeisenbergMono[\mu]}"]  \ar[d,mono,"\exists\kappa_\mu",dashed] & 
			\HH(\mu) \ar[r,epi,"{\HeisenbergEpi[\mu]}"]  \ar[d,mono,"\exists\gamma_\mu",dashed]   & 
			A\times A \ar[r] \ar[d,iso',"\exists\lambda_\mu",dashed] & 
			1
			\\
			\MumfordSES{\delta(\mu)}\ar[:] &
			1\ar[r] & 
			k^\times \ar[r,mono,"\iota_{\delta(\mu)}"]  & 
			\MumfordTheta{\delta(\mu)}\ar[r,epi,"\pi_{\delta(\mu)}"]   & 
			\MumfordH{\delta(\mu)} \ar[r]  & 
			1
		\end{tikzcd}
	\end{equation}

\end{prop}

\begin{proof}
	By assumption $\characteristic(k)\ndivides |C|$. 
	So on one hand, there is an embedding $\kappa\from C\embeds k^\times$ because 
	$k$ is algebraically closed and $C$ is cyclic. 
	On the other hand, the invariant factors $\delta(\mu)$ of $A$ is an admissible $d(A)$-tuple for $k$. 
	Pick an isomorphism $\lambda_1\from A\to K(\delta(\mu))$ to the group from   \autoref{defn:MumfordDeltaGroups}.
	We claim that $\lambda_2\from A\to \Hom(K(\delta(\mu)), k^\times)$ defined  by $a\mapsto (x\mapsto \kappa(\mu(a,\lambda_1^{-1}(x))))$ is an isomorphism.
	Indeed, it is a morphism being the composition of such maps. 
	Since $A\isom K(\delta(\mu))\isom \Hom(K(\delta(\mu)),k^\times)$, because  $\characteristic(k)\ndivides |K(\delta(\mu))|$ and these groups are finite abelian, it is enough to show that $\lambda_2$ is injective. 
	Suppose $\lambda_2(a)=0$. Then as $\kappa$ is injective and $\lambda_1$ is an isomorphism, we have $\mu(a,b)=0$ for every $b\in A$. But $\HH(\mu)$ is non-degenerate by assumption, hence so is $\mu$ by \autoref{rem:Hdef} which implies $a=0$. 
	
	We show that setting the vertical maps from \eqref{diag:HeisenbergEmbedsTOMumford} to 
	\[c\xmapsto{\kappa_\mu} \kappa(c), \qquad
	(a,b,c)\xmapsto{\gamma_\mu} (\kappa(c), \lambda_1(b), \lambda_2(a))^{-1},\qquad
	(a,b)\xmapsto{\lambda_\mu} (\lambda_1(b), \lambda_2(a))^{-1}
	\]
	gives the statement.
	$\kappa_\mu$ is a well-defined monomorphism because of the non-degeneracy of $\HH(\mu)$. 
	To see that $\gamma_\mu$ is a morphism, use $\lambda_2(a)(\lambda_1(b')) = \kappa(\mu(a,b'))$ from above together with \autoref{rem:Hdef} and \autoref{defn:MumfordDeltaGroups}: 
	\begin{align*}
		\gamma_\mu((a,b,c)*(a',b',c'))
		&= \gamma_\mu((a+a', b+b',c + \mu(a,b')+c'))
		\\&= (\kappa(c + \mu(a,b')+c'), \lambda_1(b+b'), \lambda_2(a+a'))^{-1}
		\\&=  ((\kappa(c')\kappa(c)\lambda_2(a)(\lambda_1(b')), \lambda_1(b')+\lambda_1(b), \lambda_2(a')\lambda_2(a)))^{-1}
		\\&=  ((\kappa(c'), \lambda_1(b'), \lambda_2(a'))\cdot (\kappa(c), \lambda_1(b), \lambda_2(a)))^{-1}
		\\&= \gamma_\mu((a,b,c))\cdot \gamma_\mu((a',b',c')).
	\end{align*}
	The injectivity of this map follows from that of $\kappa$, $\lambda_1$ and $\lambda_2$. 
	By above, $\lambda_\mu$ is an isomorphism. 
	The commutativity of \eqref{diag:HeisenbergEmbedsTOMumford} follows from definitions.
\end{proof}

	\subsection{Associated invertible sheaf over a product of elliptic curves}
\label{sec:algebraic-associatedGroup}

\begin{summary}
	Following \cite{mumford1966equations}, 
	we recall Mumford's theta group associated to certain invertible sheaves. 
	Then for every admissible tuple from \autoref{sec:algebraic-parametrisation}, we construct an invertible sheaf (over a suitable abelian variety of controlled dimension) such that 
	the abstract theta group of the admissible tuple and the theta group of the sheaf are isomorphic.
\end{summary}

\begin{asm}\label{asm:ampleSeparabletype}
	In this section, $k$ denotes an algebraically closed field, 
	$X$ an abelian variety over $k$, and 
	$L$ an invertible sheaf on $X$. 
	We assume that $L$ is ample and is of \emph{separable type}, i.e. $\characteristic(k)\ndivides \deg(L)$ for $\deg(L)\leteq \dim H^0(X, L)=\chi(L)$.
\end{asm}

\begin{defn}[{\cite[pp.288-289]{mumford1966equations}}]\label{defn:MumfordTheta}
	Consider the setup of \autoref{asm:ampleSeparabletype}. 
	For a closed point $x\in X(k)$, define the translation map $\tau(x)\from X(k)\to X(k), y\mapsto y+x$. 
	Define the \emph{theta group} $\MumfordTheta{L}$ to be the set of pairs $(x,\phi)$ where $x\in X(k)$ and $\phi\from L\to \tau(x)^*L$ is an isomorphism of invertible sheaves with the binary operation 
	$(y,\psi)\cdot(x,\phi)\leteq (x+y,(\tau(x)^*\psi)\circ \phi)$. 
	Let $\MumfordH{L}\leteq \{x\in X(k):L\isom \tau(x)^* L\}$. 
\end{defn}


\begin{lem}[{\cite[Theorem~1]{mumford1966equations}}]\label{lem:MumfordCentre}
	Under \autoref{asm:ampleSeparabletype}, we have  		$\Center(\MumfordTheta{L})=\{(0,\phi):\phi\in\Aut(L)\}\isom k^\times$. 
	In particular, 
	\begin{equation}\label{diag:MumfordSES}
		\begin{tikzcd}[left label]
		\MumfordSES{L}\ar[:]&
		1\ar[r] & 
		\Center(\MumfordTheta{L}) \ar[r,inclusion] & 
		\MumfordTheta{L}\ar[r,epi,"\pi_L"] & 
		\MumfordH{L} \ar[r] & 
		1
	\end{tikzcd}
	\end{equation}
	is a non-degenerate \centralByAbelian{} extension where $\pi_L\from (x,\phi)\mapsto x$.
\end{lem}

\begin{lem}[Associated admissible tuple, type]
	\label{lem:MumfordSESiso}
	Under \autoref{asm:ampleSeparabletype},  
	for every $L$,  
	there exists a unique admissible tuple $\delta=(d_1,\dots,d_t)$ for $k$ (called the \emph{type} of $L$)
	such that $\deg(L)^2=\prod_{i=1}^t d_i^2 = |\MumfordH{L}|$ and 
	there are group isomorphisms making the following diagram commutative  
	\begin{equation}\label{diag:MumfordSESiso}
		\begin{tikzcd}[left label]
			\MumfordSES{\delta}\ar[:] &
			1\ar[r] & 
			k^\times \ar[r,mono,"\iota_\delta"]  \ar[d,iso',"\exists\kappa_L",dashed] & 
			\MumfordTheta{\delta}\ar[r,epi,"\pi_\delta"]  \ar[d,iso',"\exists\gamma_L",dashed]  & 
			\MumfordH{\delta} \ar[r] \ar[d,iso',"\exists\lambda_L",dashed] & 
			1
			\\
			\MumfordSES{L}\ar[:] &
			1\ar[r] & 
			\Center(\MumfordTheta{L}) \ar[r,inclusion] & 
			\MumfordTheta{L}\ar[r,epi,"\pi_L"] & 
			\MumfordH{L} \ar[r]   & 
			1
		\end{tikzcd}
	\end{equation}
	where $\kappa_L\from c\mapsto (0,\phi_c)$ where $\phi_c$ is induced by multiplication by $c$.
\end{lem}
\begin{proof}
	The isomorphism is a combination of \cite[Corollary of Th. 1, p.294]{mumford1966equations} and \autoref{lem:MumfordCentre}. 
	Finally \cite[p.289]{mumford1966equations} shows 
	$\deg(L)^2=|\MumfordH{L}|$ but $|\MumfordH{L}|=|\MumfordH{\delta}| = (\prod_{i=1}^t d_i)^2$ from \eqref{diag:MumfordSESiso} and \autoref{defn:MumfordDeltaGroups}.
\end{proof}

To eventually apply \autoref{lem:HerisenbergToMumfordDelta}, we need to have an invertible sheaf associated to every admissible tuple. 
In other words, we want to have a converse of \autoref{lem:MumfordSESiso} in the sense of the following notion.

\begin{defn}[$s$-admissible abelian variety]\label{defn:s-admissible}
	Let $s\in\R$ and $k$ be an algebraically closed field.  
	We call an abelian variety $X$ over $k$ \emph{$s$-admissible (for $k$)}, 
	if for every integer $1\leq t\leq s$, 
	and every admissible $t$-tuple $\delta$ for $k$,
	there exists an invertible sheaf $L(\delta)$ on $X$ for which $\MumfordSES{\delta}\isom \MumfordSES{L(\delta)}$ as in \eqref{diag:MumfordSESiso}, 
	i.e.  if there exists $L\leteq L(\delta)$ of type $\delta$ satisfying \autoref{asm:ampleSeparabletype}.
\end{defn}

To see the existence of such $s$-admissible abelian varieties, we need the following facts.


\begin{lem}[{\cite[\S II.6.~Proposition(3),p.61]{mumford2008abelian}}]
	\label{lem:torsionGroupOfAbelianVariety}
	For any abelian variety $X$ over an algebraically closed field $k$ and for any integer $\characteristic(k)\ndivides d$, we have 
	$\{x\in X(k):d\cdot x=0\}\isom (\Z/d\Z)^{2\dim(X)}$.
\end{lem}

\begin{lem}[{\cite[Proposition 4,p.310]{mumford1966equations}}]
	\label{lem:MumfordHtensorPower}
	For every $d\in\Np$ and invertible sheaf $L$ on an abelian variety $X$ (over an algebraically closed field), 
	$\MumfordH{L^{\tensor d}} = \{x\in X(k):d\cdot x\in\MumfordH{L}\}$.
\end{lem}

\begin{defn}
	Let $G_1$, $G_2$ and $A$ be groups. 
	For a pair of group monomorphisms $\kappa_i\from A\to \Center(G_i)$ (for $i\in\{1,2\}$), write 
	$\kappa\from A\to G_1\times G_2,a\mapsto (\kappa_1(a),\kappa_2(a)^{-1})$. 
	Define the \emph{external central product} of $G_1$ and $G_2$ (amalgamated along $\kappa_1$ and $\kappa_2$) to be the group 
	$G_1\centralProdPlain G_2\leteq  G_1\times G_2/\im(\kappa)$.
\end{defn}

\begin{defn}
	For $i\in\{1,2\}$, let $\mathcal{F}_i$ be an $\structureSheaf{X_i}$-module and let $q_i:X_1\times X_2\to X_i$ be the natural projection. 
	Define the \emph{external tensor product} to be the $\structureSheaf{X_1\times X_2}$-module 
	$\mathcal{F}_1\extTensor \mathcal{F}_2 \leteq (q_1^*\mathcal{F}_1)\tensor_{\structureSheaf{X_1\times X_2}} (q_2^*\mathcal{F}_2)$.
\end{defn}

The notions above are linked as follows.

\begin{lem}[{\cite[Lemma~\S3.1,  p.323]{mumford1966equations}}] \label{lem:MumforThetaExtTensorOfSheaves}
	If $L_i$ are invertible sheaves on $X_i$ (over $k$) satisfying \autoref{asm:ampleSeparabletype} for $i\in\{1,2\}$,  
	then $L_1\extTensor L_2$ also satisfies \autoref{asm:ampleSeparabletype} and 
	we have the following commutative diagram 
	\[\begin{tikzcd}
		\Center(\MumfordTheta{L_1}) \centralProdPlain \Center(\MumfordTheta{L_2}) \ar[r,inclusion] \ar[d,iso,dashed,"\exists"']& 
		\MumfordTheta{L_1} \centralProdPlain \MumfordTheta{L_2}  \ar[r,epi,"\pi"] \ar[d,iso,dashed,"\exists"']& 
		\MumfordH{L_1} \times \MumfordH{L_1}   \ar[d,identity] 
		\\
		\Center(\MumfordTheta{L_1\extTensor L_2}) \ar[r,inclusion]   & 
		\MumfordTheta{L_1\extTensor L_2} \ar[r,epi,"\pi_{L_1\extTensor L_2}"] & 
		\MumfordH{L_1\extTensor L_2} 
	\end{tikzcd}\] 
	involving  short exact sequences of the form \eqref{diag:MumfordSES}
	where the central products are amalgamated along $\kappa_i\leteq \kappa_{L_i}\from k^\times\to\Center(\MumfordTheta{L_i})$ from \eqref{diag:MumfordSESiso}, and 
	$\pi$ sends the coset of $((x_1,\phi_1),(x_2,\phi_2))$ to $(x_1,x_2)$.
\end{lem}

We are ready to state a converse of \autoref{lem:MumfordSESiso}.
Recall \autoref{rem:varietyConstruction} about abelian varieties admitting a principal polarisation and examples of such varieties.
 
\begin{prop}[Existence of associated invertible sheaf]\label{prop:LforDelta}
	Let $s\in \N$. 
	Let $T_1,\dots,T_s$ be abelian varieties over an algebraically closed field $k$ such that every $T_i$ admits a principal polarisation.
	Then the abelian variety $X\leteq \prod_{i=1}^s T_i$ is $s$-admissible, cf. \autoref{defn:s-admissible}.
\end{prop}
\begin{proof}
	By \autoref{rem:varietyConstruction}, there is an ample invertible sheaf $L_i$ on $T_i$ with $\deg(L_i)=1$ for every $i\in\{1,\dots,s\}$. 
	Let $\delta = (d_1,\dots,d_t)$ be an arbitrary admissible $t$-tuple for $k$. 
	Define $d_i\leteq 1$ for all $t+1\leq i\leq s$. 
	By \autoref{lem:MumfordSESiso}, we have  $|\MumfordH{L_i}| = \deg(L_i)^2 = 1$, i.e.  
	$\MumfordH{L_i}$  is the trivial group. 
	Then \autoref{lem:MumfordHtensorPower} shows that 
	$\MumfordH{L_i^{\tensor d_i}}=\{x\in T_i(k):d_i\cdot x\in\MumfordH{L_i}=0\}\isom (\Z/d_i\Z)^2$ using  \autoref{lem:torsionGroupOfAbelianVariety} as $\characteristic(k)\ndivides d_i$. 	
	Define \[L(\delta)\leteq L_1^{\tensor d_1}\extTensor L_2^{\tensor d_2}\extTensor \cdots \extTensor L_s^{\tensor d_s},\] an ample invertible sheaf on $X$. 
	Then repeated application of \autoref{lem:MumforThetaExtTensorOfSheaves} gives 
	$\MumfordH{L(\delta)}\isom\prod_{i=1}^s (\Z/d_i\Z)^2\isom\prod_{i=1}^t (\Z/d_i\Z)^2\isom \MumfordH{\delta}$ by above.
	Thus the uniqueness part of \autoref{lem:MumfordSESiso} shows that  $\MumfordSES{\delta}\isom \MumfordSES{L(\delta)}$ and thus that $L(\delta)$ is of separable type as required.
\end{proof}

	\subsection{Associated action of Mumford's theta group}
\label{sec:algebraic-associatedAction}

\begin{summary}
	Following \cite{mumford1966equations}, we briefly recall the action of the theta groups of \autoref{sec:algebraic-associatedGroup} on the global sections. 
	This idea is used in the proof of \autoref{thm:mainBir} (even though this action is not used directly) and also motivates the crucial notion of \autoref{sec:diffActionOnLineBundle}.
\end{summary}

%
	Recall \autoref{asm:ampleSeparabletype} and \autoref{defn:MumfordTheta}. 
	Let $S$ be the set of $x\in X(k)$ such that $L(X)$ and $\tau(x)^* L(X)$ are isomorphic sheaves. 
	For $x\in S$, let $\Phi_x$ be the set of vector space isomorphisms $L(X)\to \tau(x)^* L(X)$. 
	Similarly to $\MumfordTheta{L}$, we endow $\Phi\leteq\bigcup_{x\in S} \Phi_x$ with a group structure by 
	$g\cdot f\leteq (\tau(x)^* g)\circ f\in\Phi_{x+y}$ for 
	$f\in\Phi_x$, $g\in\Phi_y$.
	%

	These groups sit in the following commutative diagram 
	\begin{equation}\label{diag:MumfordHomAction}
		\begin{tikzcd}[left label]
			\MumfordSES{L}\ar[:]
			& 1\ar[r] 
			& \Center(\MumfordTheta{L}) \ar[r,inclusion] \ar[d,mono, "\exists\sigma",dashed]
			& \MumfordTheta{L}\ar[r,epi,"\pi_L"] \ar[d,mono, "\exists\rho",dashed]
			& \MumfordH{L} \ar[r] \ar[d,mono,"\exists\tau",dashed]
			& 1
			\\
			& 1\ar[r]
			& \Phi_0 \ar[r,inclusion] 
			& \Phi \ar[r,epi,"\pi"] 
			& \tau(S) \ar[r]
			& 1
	\end{tikzcd}\end{equation}
	where the vertical morphisms are injective and 
	$\rho$ is given by $(x,\phi)\mapsto \phi_X$ (the map on the global sections), 
	$\sigma$ is its restriction and
	$\pi$ is given by $f\mapsto \tau(x)$ where $f\in\Phi_x$.	

	Actually, \cite[Theorem~2]{mumford1966equations} shows that the composition  $\MumfordTheta{L}\xrightarrow{\rho}\Phi\xrightarrow{\theta}\Aut(L(X))$ is an irreducible representation of $\MumfordTheta{L}$ where $\theta$ is given by $f\mapsto \tau(-x)^* f$ for $f\in\Phi_x$.
	We are interested in the injective map $\rho$ from \eqref{diag:MumfordHomAction} instead of Mumford's representation as $\rho$ visibly keeps the $X(k)$-part of elements of $\MumfordTheta{L}$ via the map $\pi$.

	\subsection{Uniformisation of the action}
\label{sec:birationalUniformisation}
\begin{summary}
	The underlying space of the action of \eqref{diag:MumfordHomAction} depends on the parameterising admissible tuple of \autoref{sec:algebraic-parametrisation}, thus on the invertible sheaf.  
	We show that this action induces one on a `uniform space', a space which depends only on the variety over which the sheaf is defined and not on the sheaf itself. In terms of the admissible tuples, the uniform space depends only on the length of the tuple and not on its entries. 
	The main observation for this is that \eqref{diag:MumfordHomAction} could be translated to an action on the corresponding line bundle, and that every line bundle is birationally equivalent to the trivial one. 
\end{summary}

\begin{defn}\label{defn:Birp}
	Let $p\from E\to X$ be a vector bundle over a variety $X$.
	Define $\Bir_p(E)$ to be the group of birational automorphisms $\phi$ of $E$ for which there exists (a unique) birational automorphism $p_*\phi$ of $X$ such that $p\circ \phi = (p_*\phi)\circ p$. 
	Let $\Bir_p(X)\leteq \{p_*\phi\in\Bir(X):\phi\in\Bir_p(E)\}$. 
	Define $\Bir_p^\id(E)\leteq \ker(p_*)\leq \Bir_p(E)$.
\end{defn}

The next statement builds on \cite{zarhin2014}.

\begin{prop}[Uniformisation]\label{lem:uniformisationBir}
	Under \autoref{asm:ampleSeparabletype},   
	there exists the following commutative diagram of groups with injective vertical morphisms.
	\begin{equation}\label{diag:BirUniformisation}
		\begin{tikzcd}[left label,column sep=small]
			\MumfordSES{L}\ar[:]
			& 1\ar[r] 
			& \Center(\MumfordTheta{L}) \ar[r,inclusion] \ar[d,mono, "\exists\sigma_L",dashed]
			& \MumfordTheta{L}\ar[r,epi,"\pi_L"] \ar[d,mono, "\exists\rho_L",dashed]
			& \MumfordH{L} \ar[r] \ar[d,mono,"\exists\tau",dashed]
			& 1
			\\
			\Bir_p\ar[:]
			& 1\ar[r]
			& \Bir_p^\id(X\times \Affine_k^1) \ar[r,inclusion] 
			& \Bir_p(X\times \Affine_k^1) \ar[r,epi,"p_*"] 
			& \Bir(X) \ar[r]
			& 1
		\end{tikzcd}
	\end{equation}
	where $p\from X\times\Affine_k^1\to X$ is the trivial line bundle.
\end{prop}

\begin{proof}
	Use the notation of \cite[II.6~Invertible~Sheaves]{hartshorne}. 
	Let $\mathcalcurly{K}$ be the constant $\structureSheaf{X}$-module of rational functions.
	Then for some Cartier divisor $D$, we have $L=\invertibleSheaf{D}$, a sub-$\structureSheaf{X}$-module of $\mathcalcurly{K}$. 
	Pick $(x,\phi)\in\MumfordTheta{L}$.  
	Now the isomorphism $\phi\from \invertibleSheaf{D}\to\tau(x)^*\invertibleSheaf{ D}$ is given by multiplication by some rational function $f_\phi\in\Gamma(X,\mathcalcurly{K}^*)$. 
	Define 
	\begin{equation}\label{eq:rhoUniformisationBir}
		\rho_L\from (x,\phi)\mapsto ((z,t)\mapsto (z+x, f_\phi(z)t)
	\end{equation}	on the non-empty open set where $f_\phi$ is defined.  
	The injectivity of this map is clear, so we only need to check that it is a group morphism.
	Note that 
	\begin{equation}\label{eq:fPropertyBir}
		f_{\tau(x)^*\psi\circ \phi}(z) = f_{\tau(x)^*\psi}(z)f_{\phi}(z) = f_{\psi}(z+x)f_{\phi}(z),
	\end{equation}
	cf. \eqref{eq:fPropertyHolo} for analogy. 
	Then 
	\begin{equation}\label{eq:rhoMorphismCheckBir}
		\begin{aligned}
			\rho_L((y,\psi)\cdot (x,\phi))&=
			(z,t)\mapsto\rho_L((x+y, \tau(x)^*\psi\circ \phi))(z,t)\\&=
			(z,t)\mapsto(z+x+y, f_{\tau(x)^*\psi\circ \phi}(z)t)\\&=
			(z,t)\mapsto(z+x+y, f_\psi(z+x) f_\phi(z)t)\\&=
			(z,t)\mapsto\rho_L(y,\psi) (z+x, f_\phi(z)t)\\&=
			\rho_L(y,\psi)\circ\rho_L(x,\phi),
		\end{aligned}
	\end{equation}
	so $\rho_L$ is indeed a morphism, cf. \eqref{eq:rhoProperty} for analogy. 
	We define $\sigma_L$ as the restriction of $\rho_L$ to $\Center(\MumfordTheta{L})$. This indeed maps to $\Bir_p^\id(X\times \Affine_k^1)$ by \autoref{lem:MumfordCentre}. 
	The map $\tau$ is defined by  \autoref{defn:MumfordTheta}. 
	These maps are injective and the commutativity of the diagram is a straightforward consequence of these definitions.
\end{proof}

	\subsection{Proof of \autoref{thm:mainBir}}
\label{sec:proofA}
\begin{summary}
	Applying the results of the previous section, we prove one of our main results, \autoref{thm:mainBir} (from \autopageref{thm:mainBir}). 
	We embed every nilpotent group $G$ of class at most $2$ and rank at most $r$ to a theta group corresponding to an admissible tuple of bounded length. 
	Then the uniformisation of the corresponding action of the theta group over the field $k$ gives an action on a fixed variety $X_{r,k}$ (depending on $r$ and the field). 
\end{summary}

\begin{proof}[Proof of \autoref{thm:mainBir}]
	We show that $X_{r,k}=\prod_{i=1}^r\prod_{j=1}^{\lfloor r/2\rfloor}T_{i,j} \times Y$ from \autoref{rem:varietyConstruction} satisfies the statement. 
	Write $T_i\leteq \prod_{j=1}^{\lfloor r/2\rfloor}T_{i,j}$, an $(r/2)$-admissible abelian variety by \autoref{prop:LforDelta}.  
	
	Let $G$ be any group as in the statement. 
	Apply \autoref{thm:embedToH} to get non-degenerate Heisenberg groups $\HH(\mu_i\from A_i\times A_i\to C_i)$ with cyclic centre for $1\leq i\leq n\leteq d(\Center(G))$ and an embedding $\delta$ of $G$ to the product of these groups. 
	Note that as the rank of $G$ is at most $r$, we have $n\leq r$ by \autoref{defn:rankGroup}.
	Since $\characteristic(k)\ndivides |\HH(\mu_i)|$ by \autoref{thm:embedToH}, we 
	can apply \autoref{lem:HerisenbergToMumfordDelta} to embed each Heisenberg group further into the abstract theta group $\MumfordTheta{\delta_i}$ via $\gamma_{\mu_i}$ for suitable admissible $t_i$-tuple $\delta_i\leteq \delta(\mu_i)$. Note that $t_i= d(A_i)\leq \frac{1}{2} d(G)\leq \frac{r}{2}\eqlet s$ by \autoref{thm:embedToH} and \autoref{lem:HerisenbergToMumfordDelta}.
	The abelian varieties $T_i$ are $s$-admissible as noted in the beginning. 
	Hence for each for each admissible $t_i$-tuple $\delta(\mu_i)$, \autoref{defn:s-admissible} gives  
	an invertible sheaf $L_i\leteq L(\delta_i)$ on $T_i$ satisfying \autoref{asm:ampleSeparabletype} whose theta group $\MumfordTheta{L_i}$ is isomorphic to the abstract theta group $\MumfordTheta{\delta_i}$ via $\gamma_i\leteq \gamma_{L_i}$ from \eqref{diag:MumfordSESiso}. 
	These groups act birationally and faithfully on $T_i\times \Affine_k^1$ by \autoref{lem:uniformisationBir} via $\rho_i\leteq \rho_{L_i}$. 
	
	Hence, as indicated in \eqref{diag:proofBir}, the product $\prod_{i=1}^n \HH(\mu_i)$ of these groups Heisenberg group acts birationally and faithfully on the product space $\prod_{i=1}^n (T_i\times \Affine_k^{1})$. 
	This induces a birational and faithful action on $\left(\prod_{i=1}^r T_i\right) \times \Affine_k^{r}$ as $n\leq r$, and hence one on $X_{r,k}=\prod_{i=1}^r T_i \times Y$ as the rational variety $Y$ is birationally equivalent to $\Affine_k^{r}\times \Affine_k^{\dim(Y)-r}$ by definition.
	\begin{equation}\label{diag:proofBir}
		\begin{tikzcd}[displaystyle,column sep=5mm]
		G\ar[r,mono,"\text{\eqref{diag:embeddingToHeisenberg}}"',"\delta"] \ar[d,mono,dashed,"\exists"]
		& \prod_{i=1}^{n} \HH(\mu_i) \ar[r,mono,"\prod_i \gamma_{\mu_i}","\text{\eqref{diag:HeisenbergEmbedsTOMumford}}"']
		& \prod_{i=1}^{n} \MumfordTheta{\delta_i} \ar[r,iso',"\prod_i\gamma_i","\text{\eqref{diag:MumfordSESiso}}"' {yshift=-3pt}]
		& \prod_{i=1}^{n} \MumfordTheta{L_i} 
		\ar[d,mono,longarrow,"\prod_i \rho_i","\text{\eqref{diag:BirUniformisation}}"']
		\\
		\Bir(X_{r,k})
		& \Bir(\prod_{i=1}^r T_i \times \Affine_k^{r}) \ar[l,mono]
		& \Bir(\prod_{i=1}^n T_i \times \Affine_k^{n}) \ar[l,mono]
		& \prod_{i=1}^{n} \Bir(T_i\times \Affine_k^1) \ar[l,mono]
	\end{tikzcd}
	\end{equation}
	The composition of the monomorphisms above gives the faithful birational action of $G$ on $X_{r,k}$ as stated.
\end{proof}

\begin{rem}\label{rem:mainTheoremAdmissible}
	Replacing $T_i$ by an arbitrary $(r/2)$-admissible variety, 
	the proof above shows that the more general variety $X=\prod_{i=1}^r T_i \times Y$ also satisfies the conclusion of \autoref{thm:mainBir}.
\end{rem}

\section{Compact manifolds and \autoref{thm:mainDiff}}\label{sec:diff}
\begin{summary}
	In this section, we prove the second main statement of the paper: \autoref{thm:mainDiff} from \autopageref{thm:mainDiff}.
\end{summary}
	\subsection{\CentralByAbelian{} extensions via isotropic sublattice data}
\label{sec:parametrisationHermitianForm}
\begin{summary}
We introduce the notion of isotropic sublattice data consisting of a complex vector space with a Hermitian form, two lattices and a real structure. 
This data induces a finite Heisenberg group. 
Using \cite{Heisenberg} by the author, we show that given a non-degenerate finite Heisenberg group with cyclic centre, we can construct an isotropic sublattice data whose induced Heisenberg group is isomorphic to the given one.
\end{summary}

\begin{defn}
	For a commutative ring $Q$, define a ring $Q[i] \leteq Q[x]/(x^2+1)$ with $i\leteq x+(x^2+1)\in Q[i]$ and maps
	$\sigma\from Q[i]\to Q,q+iq'\mapsto q-iq'$ (conjugation) 
	and $\Im\from Q[i]\to Q,q+iq'\mapsto q'$ (the imaginary part). 
	For a $Q[i]$-module $M$, we call a map $h\from M\times M\to Q[i]$ a \emph{Hermitian form on $M$ over $Q[i]$} if $h$ is $Q[i]$-linear in the first argument and $h$ is $\sigma$-conjugate symmetric (i.e. $h(m,m')=\sigma(h(m',m))$). 
	Write $\Im h\leteq \Im\circ h\from M\times M\to Q$.
\end{defn}

For an $R$-bilinear map $f\from A\times B\to C$, by $f(A,B)\leq C$ we mean the $R$-module generated by $\{f(a,b):a\in A,b\in B\}$.

\begin{defn}\label{defn:isotropicSublatticeData}
	An \emph{isotropic sublattice data} $\mathfrak{D}$ is a tuple  $(h,V,V_\Re,L_\Re,\Lambda_\Re, \Gamma)$ where 
	\begin{itemize}
		\item $V$ is a $\C$-vector space; 
		\item $h\from V\times V\to \C$ is a Hermitian form;
		\item $V_\Re\leq V$ is an $\R$-subspace satisfying $V=V_\Re\oplus iV_\Re$ and $\Im h(V_\Re,V_\Re)=0$ (which we call an \emph{isotropic real structure});
		\item $\Lambda_\Re\leq L_\Re$ are both \emph{lattices} in $V_\Re$ 
		(i.e. free abelian groups of rank $\dim_\R(V_\Re)$ generating $V_\Re$ as an $\R$-vector space) 	
 		such that 
		$\Im h(L_\Re,i\Lambda_\Re)=\Z$ whenever $h(V,V)\neq 0$;
		\item $\Gamma$ is a group satisfying  $\Z\leq\Gamma\leq \R$ and $\Im h (L_\Re, iL_\Re)\leq \Gamma$.
	\end{itemize}
	In this case, we write $V_\Im\leteq iV_\Re$, $L_\Im\leteq iL_\Re$, $L\leteq L_\Re\oplus L_\Im$ and 
	$\Lambda_\Im\leteq i\Lambda_\Re$, $\Lambda \leteq \Lambda_\Re\oplus \Lambda_\Im$. 
	For $X\in\{V,L,\Lambda\}$ and $x\in X$, write $x_\Re\in X_\Re$, $x_\Im\in X_\Im$ for the unique elements satisfying $x=x_\Re+x_\Im$.
\end{defn}

\begin{rem}
	$\Lambda$ and $L$ are both lattices in $V$. 
	Since $ h(v,v')= h(iv,iv')$, we have $\Im(V_\Im,V_\Im)=0$, so 
	$\Im h(L,\Lambda) = \Im h(L_\Re,\Lambda_\Im) = \Im h(L_\Im,\Lambda_\Re)$ and $\Im h(L,L)=\Im h (L_\Re, L_\Im)$.
\end{rem}

\begin{lem}\label{lem:quotientHermitianForm}
	Every isotropic sublattice data $\mathfrak{D}=(h,V,V_\Re,L_\Re,\Lambda_\Re, \Gamma)$ induces a (unique) $\Z$-bilinear map $\mu_\mathfrak{D}$ making the following diagram commutative. 	
	\begin{equation}\label{eq:HermitianFormDiscretised}
		\begin{tikzcd}
		1 \ar[r]
		& \Lambda_\Re\times \Lambda_\Im \ar[r,inclusion] \ar[d,"\Im h|_{\Lambda_\Re\times \Lambda_\Im}"]
		& L_\Re\times L_\Im \ar[r] \ar[d,"\Im h|_{L_\Re\times L_\Im}"]
		& (L_\Re / \Lambda_\Re)\times (L_\Im / \Lambda_\Im) \ar[r] \ar[d,dashed,"\exists \mu_\mathfrak{D}"]
		& 1
		\\
		1 \ar[r]
		& \Z \ar[r,inclusion]
		& \Gamma \ar[r]
		& \Gamma/\Z \ar[r]
		& 1 
	\end{tikzcd}
	\end{equation}
\end{lem}
\begin{rem}
	We have $1\to 
	\HH(\Im h|_{\Lambda_\Re\times \Lambda_\Im})\to 
	\HH(\Im h|_{L_\Re\times L_\Im}) \to
	\HH(\mu_\mathfrak{D})\to
	1 $ by \autoref{lem:Hfunctorial} and the $3\times 3$-lemma.
\end{rem}

\begin{proof}
	If $h$ is trivial, then so is $\mu_\mathfrak{D}$. 
	Otherwise, we check that $(l_\Re+\Lambda_\Re, l_\Im+\Lambda_\Im)\mapsto \Im h(l_\Re,l_\Im)+\Z$ is well-defined.
	Pick $l_\Re,l_\Re'\in L_\Re$ so that $l_\Re-l_\Re'\in \Lambda_\Re$, and similarly 
	$l_\Im,l_\Im'\in L_\Im$ so that $l_\Im-l_\Im'\in \Lambda_\Im$. 
	Then $\Im h(l_\Re,l_\Im)-\Im h(l_\Re',l_\Im') =
	\Im h(l_\Re-l_\Re', l_\Im) + \Im h(l_\Re',l_\Im-l_\Im')
	\in \Im h(\Lambda_\Re, L_\Im)$ +  $\Im h(L_\Re,\Lambda_\Im)
	= \Im h(L_\Re,\Lambda_\Im)=\Z$ because 
	$\Im h(v,iv') = \Im h(v',iv)$ for any $v,v'\in V$.
	The $\Z$-bilinearity of the map follows from the definition of $h$ being a Hermitian form.
\end{proof}

We use the following result stating that non-degenerate Heisenberg groups arise from Hermitian forms. We write $\Z_c$ for the ring $\Z/(c)$ where $c\in \Z$.
\begin{lem}[\cite{Heisenberg}]\label{lem:HermitianFromBilinear}
	Let $\HH(\mu\from A\times B\to C)$ be finite and non-degenerate, and 
	let  $c=|\mu(A,B)|$.
	Then $M=A\times B$ can be endowed with a $\Z_c[i]$-module structure and there exists
	a hermitian form $h_M\from M\times M\to \Z_c[i]$, 
	a group monomorphism $\phi\from \Z_c\to C$ 
	and $\alpha\in A$ of order $c$ such that
	$B=iA$,  
	$\Im h_M(\alpha,i\alpha)=1\in \Z_c$, 
	$\Im h_M(A,A)=\Im h_M(B,B)=0$
	and $\mu(a,b)=\phi(\Im h_M(a,b))$ for every $(a,b)\in A\times B$.
\end{lem}
\begin{proof}
	This follows from \cite[Proposition 3.12]{Heisenberg} using \cite[Remarks 3.13-15]{Heisenberg} after possibly applying an automorphism of $\Z_c$.
\end{proof}

\begin{prop}\label{prop:HermitianFormFromHeisenberg}
	For every finite non-degenerate $\HH(\mu\from A\times B\to C)$ with cyclic centre, there exists an isotropic sublattice data $\mathfrak{D}=(h,V,V_\Re,L_\Re,\Lambda_\Re, \Gamma\leteq \frac{1}{|C|}\Z)$ where $\dim_\C(V)=d(A)=d(B)$ that induces the commutative diagram 
	\begin{equation}\label{diag:HeisenbergEmbedsTOHermitian}
		\begin{tikzcd}[left label,column sep=scriptsize]
			\HeisenbergFunctor(\mu) \ar[:]& 
			1\ar[r] & 
			C\ar[r,mono,"{\HeisenbergMono[\mu]}"]  \ar[d,mono,iso',"\exists\kappa_\mathfrak{D}",dashed] & 
			\HH(\mu) \ar[r,epi,"{\HeisenbergEpi[\mu]}"]  \ar[d,mono,iso',"\exists\gamma_\mathfrak{D}",dashed]   & 
			A\times B \ar[r] \ar[d,iso',"\exists\lambda_\mathfrak{D}",dashed] & 
			1
			\\
			\HeisenbergFunctor(\mu_\mathfrak{D})\ar[:] &
			1\ar[r] & 
			\Gamma/\Z \ar[r,mono,epi,"{\HeisenbergMono[{\mu_\mathfrak{D}}]}"]  & 
			\HH(\mu_\mathfrak{D}) \ar[r,epi,"{\HeisenbergEpi[{\mu_\mathfrak{D}}]}"]   & 
			(L_\Re/\Lambda_\Re)\times (L_\Im/\Lambda_\Im)  \ar[r]  & 
			1
		\end{tikzcd}
	\end{equation}
	where $\mu_\mathfrak{D}$ is from \autoref{lem:quotientHermitianForm}.
\end{prop}
\begin{proof} 
	If $A$ is the trivial group, then we may take $0=V=V_\Re=L_\Re=\Lambda_\Re$ and any isomorphism between $C$ and $\Gamma/\Z$. So we may assume that $A\isom B$ is not trivial.
	
	Let $M$ be the $\Z_c[i]$-module, $h_M\from M\times M\to \Z_c[i]$ be the Hermitian form, $\phi\from \Z_n\embeds C$ and (the non-trivial) $\alpha\in A$ from \autoref{lem:HermitianFromBilinear}. 
	Since $|\alpha|=c$ equals the exponent of $A$, we can extend $\alpha$ to a minimal $\Z_c$-module generating set $S=\{\alpha_1\leteq \alpha,\alpha_2,\dots,\alpha_n\}$ of $A$. 
	Note that $S$ is a minimal $\Z_c[i]$-generating set of $M$. 
	For every $\alpha_j\in S$, assign a formal symbol $\bar\alpha_j$.  
	Let \[L_\Re = \bigoplus_{j=1}^n \Z \bar\alpha_j\] be the free $\Z$-module on $\bar S\leteq\{\bar\alpha_j:\alpha_j\in S\}$. 
	Set $L_\Im \leteq iL_\Re$, $L\leteq L_\Re\oplus L_\Im$ and endow it with the natural $\Z[i]$-module structure. 
	Let $\pi_L\from L\to M$ be the $\Z[i]$-module projection defined by  $\bar\alpha_k\mapsto\alpha_k$.
	Let $\Lambda\leteq \ker(\pi_L)\leq L$, 
	\[\Lambda_\Re\leteq \Lambda\cap L_\Re,\qquad \Lambda_\Im\leteq i\Lambda_\Re.\] 
	Define a $d(A)$-dimensional $\C$-vector space by 
	\[V\leteq \R\tensor_\Z L=\C\tensor_\Z L_\Im,\] 
	and an $\R$-subspace by \[V_\Re\leteq \R\tensor_\Z L_\Re.\] 
	To define a suitable hermitian form on $V$, let $\pi\from \Z[i]\to \Z_c[i],x\mapsto x+c\Z[i]$ be the natural projection. 
	For $j\leq k$, pick $h_L(\bar\alpha_j,\bar\alpha_k)\in\{r+is\in\Z[i]:0\leq r,s<c\}$ so that  	
	$\pi(h_L(\bar\alpha_j,\bar\alpha_k)) = h_M(\pi_L(\alpha_j),\pi_L(\alpha_k))\in\Z_c[i]
	$. 
	Note that $\Im h_L(\bar\alpha_1,i\bar\alpha_1)=1$ by the choice of $\alpha_1=\alpha$ from \autoref{lem:HermitianFromBilinear}. 
	Extend this to a Hermitian form $h_L\from L\times L\to \Z[i]$ and define 
	a Hermitian form 
	\[h\leteq \tfrac{1}{c}\tensor h_L \from V\times V\to \C.\]  
	
	We claim that $\mathfrak{D}\leteq (h,V,V_\Re,L_\Re,\Lambda_\Re, \frac{1}{|C|}\Z)$ is an isotropic sublattice data for the objects constructed above. 
	Indeed, first note that $V=V_\Re\oplus iV_\Re$ by construction and that $\Lambda_\Re\leq L_\Re$ are both latices in $V_\Re$. 
	As $\pi_L(L_\Re)=A$ by definition and 
	$h_M(A,A)=0$ by \autoref{lem:HermitianFromBilinear}, the choice of the lifts in the definition of $h_L$ gives $h_L(L_\Re,L_\Re)=0$. 
	Thus $h(V_\Re,V_\Re)=\R\tensor h_L(L_\Re,L_\Re)=0$ (upon identifying $L_\Re$ by $1\tensor L_\Re$). 
	
	Next, we claim that $\Im h_L(L,\Lambda)=c\Z$. 
	Indeed, on one hand 
	$\pi(\Im h_L(L,\Lambda)) 
	=\Im h_M(M,0)
	=0$ by construction, so $\Im h_L(L,\Lambda)\subseteq \ker(\pi|_\Z)=c\Z$. 
	On the other hand, the order of $\alpha_1$ in $M$ is $c$, so 
	$\pi_L(c \bar\alpha_1)=c\alpha_1=0\in M$, hence $c\bar\alpha_1\in \Lambda$, thus 
	$\Im h(\bar\alpha_1,c\bar\alpha_1)=c$ and the claim follows. 
	Thus by the definition of $h$, 
	$\Im h(L_\Re,i\Lambda_\Re) 
	= \frac{1}{c}\Im h_L(L_\Re,i\Lambda_\Re)
	= \frac{1}{c}\Im h_L(L,\Lambda)
	=\Z
	$ as needed for \autoref{defn:isotropicSublatticeData}.
	
	Recall that $\Im h_L(\bar\alpha_1,i\bar\alpha_1)=1$ by construction, so $\Im h_L(L,L)=\Z$. 
	This shows that $\Im h(L,L) = \frac{1}{c} \Im h_L(L,L) = \frac{1}{c} \Z \subseteq \frac{1}{|C|}\Z=\Gamma$. 
	Hence $\mathfrak{D}$ is indeed an isotropic sublattice data.

	Finally, we prove the existence of  \eqref{diag:HeisenbergEmbedsTOHermitian}. The construction above can be summarised as the first two rows of the following commutative diagram.
	\begin{equation}\label{diag:HermitianForms}
		\begin{tikzcd}
			1\ar[r]
			& \Lambda_\Re\times \Lambda_\Im\ar[r,inclusion] \ar[d,"\Im h_L"]
			& L_\Re\times L_\Im\ar[r,epi,"\pi_L"]  \ar[d,epi,"\Im h_L"]
			& A\times B \ar[r] \ar[d,epi,"\Im h_M"]
			&1
			\\
			1\ar[r]
			& c\Z\ar[r,inclusion] \ar[d,iso',"\cdot\frac{1}{c}"]
			& \Z\ar[r,epi,"\pi"] \ar[d,mono,"\cdot\frac{1}{c}"]
			& \Z_c \ar[r] \ar[d,mono,"\phi"]
			& 1
			\\
			1\ar[r]
			& \Z \ar[r,inclusion]
			& \Gamma \ar[r,epi,dashed,"\pi_\Gamma"]
			& C \ar[r]
			& 1
		\end{tikzcd}
	\end{equation} 
	The vertical maps to the bottom row are multiplication by $\frac{1}{c}$ and the $\phi\from \Z_n\embeds C$ from the beginning of the proof given by \autoref{lem:HermitianFromBilinear}, and $\pi_\Gamma$ is a morphism of cyclic groups making the diagram commutative. 
	Recall that $h=\frac{1}{c}\tensor h_L$ by definition and 
	that $\mu=\phi\circ \Im h_M|_{A\times B}$ by \autoref{lem:HermitianFromBilinear}. 
	Thus considering the composition of the vertical maps of \eqref{diag:HermitianForms} and the uniqueness of \eqref{eq:HermitianFormDiscretised} induces the following commutative diagram.
	\[\begin{tikzcd}[column sep=4em]
		A\times B \ar[d,"\mu=\phi\circ \Im h_M|_{A\times B}"'] \ar[r,dashed,iso',"{\exists\lambda_\mathfrak{D}}"]
		& L_\Re/\Lambda_\Re \times L_\Im/\Lambda_\Im \ar[d,"\mu_\mathfrak{D}"] 
		\\
		C  \ar[r,dashed,iso',"{\exists\kappa_\mathfrak{D}}"]
		& \Gamma/\Z
	\end{tikzcd}\]
	Applying \autoref{lem:Hfunctorial} to this diagram gives \eqref{diag:HeisenbergEmbedsTOHermitian} from the statement.
\end{proof}

	\subsection{Associated holomorphic line bundle over a complex torus}
\label{sec:associatedLineBundle}

\begin{summary}
The Appell--Humbert theorem \cite[Theorem~2.2.3]{ComplexAbelianVarieties} classifies every holomorphic line bundle over a complex torus using Riemann forms and semi-characters. 
We use this classification to define a holomorphic line bundle over a complex torus associated to an isotropic sublattice data. 
We review the necessary details of the construction of \cite[\S1-2]{ComplexAbelianVarieties} via universal covers, as in 
\autoref{sec:diffActionOnLineBundle} we will use a slight extension of this construction to define an action.
\end{summary}

\begin{defn}\label{defn:mapsInducedbyD}
	For every isotropic sublattice data $\mathfrak{D}=(h,V,V_\Re,L_\Re,\Lambda_\Re, \Gamma)$, define the following maps.
	\begin{align*}
		\chi_\mathfrak{D}\from L&\to \complexCircle, & v &\mapsto \exp(\pi i\Im  h(l_\Re,l_\Im))
		\\
		f_\mathfrak{D}\from L\times V&\to \C^\times, & (l,v) &\mapsto \chi_\mathfrak{D}(l)\exp(\pi H(v,l)+\tfrac{\pi}{2} H(l,l))
		\\
		\sigma_\mathfrak{D} \from \Gamma&\to \Biholo(V\times \C), &
		c&\mapsto ((v,z)\mapsto (v,\exp(-2\pi i c)z))
		\\
		\tau_\mathfrak{D}\from L&\to \Biholo(V), &
		l&\mapsto (v\mapsto l+v)
		\\
		\rho_\mathfrak{D} \from  L&\to \Biholo(V\times \C), & 
		l&\mapsto ((v,z) \mapsto (\tau_\mathfrak{D}(l)(v),f_\mathfrak{D}(l,v)z))
	\end{align*}
	Here $\complexCircle\leteq \{z\in \C:|z|=1\}$ is the complex unit circle, 
	and $\Biholo(U)$ is the set of biholomorphisms on a complex space $U$, i.e. the set of bijective holomorphic functions $f\from U\to U$ such that $f^{-1}$ is also holomorphic.
\end{defn}

We need some computational results slightly extending \cite{ComplexAbelianVarieties} that will also be used later in \autoref{sec:diffActionOnLineBundle}.
\begin{lem}
	If $\mathfrak{D}=(h,V,V_\Re,L_\Re,\Lambda_\Re, \Gamma)$ is an isotropic sublattice data, $v\in V$ and $l,l'\in L$, then the maps from \autoref{defn:mapsInducedbyD} satisfy the following identities.
	\begin{align}
		\chi_\mathfrak{D}(l+l') &= \chi_\mathfrak{D}(l) \cdot \chi_\mathfrak{D}(l')\cdot  \exp(\pi i\Im h(l',l)) \cdot \exp(2\pi i\Im h(l_\Re,l_\Im'))
		\label{eq:chiProperty}
		\\
		f(l+l',v) &= f_\mathfrak{D}(l,l'+v)\cdot f(l',v)\cdot \exp(2\pi i\Im h(l_\Re,l'_\Im))
		\label{eq:fPropertyHolo}
		\\
		\rho_\mathfrak{D}(l+l') &= \rho_\mathfrak{D}(l) \circ \rho_\mathfrak{D}(l') \circ \sigma_\mathfrak{D}(-\Im h(l_\Re,l'_\Im))
		\label{eq:rhoProperty}
	\end{align}

\end{lem}
\begin{proof}
	First note that $\Im h\from V\times V\to \R$ is an alternating bilinear map satisfying $\Im h(V_\Re,V_\Re)=\Im h (V_\Im,V_\Im)=0$. 
	This implies that 
	$\Im h (l',l) 
	= \Im h(l'_\Re+l'_\Im, l_\Re + l_\Im)
	= \Im h(l'_\Re,l_\Im) - \Im h(l_\Re, l'_\Im)$, 
	hence
	\begin{align*}
		\Im h(l_\Re+l'_\Re, l_\Im+l'_\Im) 
		&= \Im h(l_\Re, l_\Im) + 
		\Im h(l'_\Re, l'_\Im) +
		\Im h(l_\Re, l'_\Im) + 
		\Im h(l'_\Re, l_\Im)
		\\&=\Im h(l_\Re, l_\Im) + 
		\Im h(l'_\Re, l'_\Im) +
		\Im h (l',l)  + 
		2\Im h(l_\Re, l'_\Im),
	\end{align*}
	thus applying the exponential function gives \eqref{eq:chiProperty}.
	
	Next note that $2h(l',l) = h(l,l') + h(l',l) + 2i\Im h(l',l)$, so 
	$2i\Im h(l',l) + h(l+l',l+l')= 2h(l',l) + h(l,l) + h(l',l')$. 
	Using this and \eqref{eq:chiProperty} we obtain
	\begin{align*}
		f_\mathfrak{D}(l+l',v) &=
		\hat\chi(l+l') \exp(\pi h(v,l+l')+\tfrac{\pi}{2} h(l+l',l+l'))
		\\&=
		\chi_\mathfrak{D}(l) \chi_\mathfrak{D}(l') 
		\exp(2\pi i\Im h(l_\Re,l'_\Im)) 
		\exp(\pi h(v,l+l'))
		\alignbreak \cdot
		\exp(\pi i\Im h(l',l)) 
		\exp(\tfrac{\pi}{2} h(l+l',l+l'))
		\\&=
		\chi_\mathfrak{D}(l) 	\exp(\pi h(v,l)) 	\exp(\pi h(l',l))	\exp(\tfrac{\pi}{2} h(l,l))
		\alignbreak \cdot
		\chi_\mathfrak{D}(l') 	\exp(\pi h(v,l'))	\exp(\tfrac{\pi}{2} h(l',l'))
		\cdot \exp(2\pi i\Im h(l_\Re,l'_\Im))
		\\&=
		f_\mathfrak{D}(l,l'+v)\cdot f(l',v)\cdot \exp(2\pi i\Im h(l_\Re,l'_\Im))
	\end{align*}
	as stated in \eqref{eq:fPropertyHolo}.
	
	Finally, we use the previous parts to prove \eqref{eq:rhoProperty} as follows.
	\begin{align*}
		\rho_\mathfrak{D}(l+l')
		&=
		(v,z)\mapsto 
		(l+l'+v,f_\mathfrak{D}(l+l',v)z)
		\\&=
		(v,z)\mapsto (l+l'+v,f_\mathfrak{D}(l,v+l')\cdot f(l',v)\cdot \exp(2\pi i\Im h(l_\Re,l'_\Im))z)
		\\&=
		(v,z)\mapsto \rho_\mathfrak{D}(l)(l'+v, f(l',v)\cdot \exp(2\pi i\Im h(l_\Re,l'_\Im))z)
		\\&=
		(v,z)\mapsto \rho_\mathfrak{D}(l)(\rho_\mathfrak{D}(l')(v,\exp(2\pi i\Im h(l_\Re,l'_\Im))z))
		\\&=
		\rho_\mathfrak{D}(l) \circ \rho_\mathfrak{D}(l') \circ \hat\sigma(-\Im h(l_\Re,l'_\Im))
		\qedhere
	\end{align*}
\end{proof}

\begin{defn}
	Let $\tau\from M\to \Aut(X)$ be a group morphism (i.e. a group action on $X$). 
	Write $X/\tau\leteq \{\{\tau(m)(x):m\in M\}:x\in X\}$, the \emph{orbit space}.
	
	Say $\tau$ is \emph{free}, if $\tau(m)(x)=x$ implies that $m=1$, i.e. when all stabilisers are trivial apart from that of $1\in M$.
\end{defn}
\begin{thm}[Quotient Manifold Theorem, finite group action]\label{lem:quotientManifoldCohomology}
	Let $X$ be a smooth real (respectively complex) manifold, 
	$M$ be a finite group, 
	$\tau\from M\to \Aut(X)$ be a free action. 
	Then $M/\tau$ has a unique smooth (respectively complex) manifold structure making 
	the natural map 
	$q\from X\to X/\tau, x\mapsto \{\tau(m)(x):m\in M\}$ a smooth (respectively holomorphic) normal $|M|$-sheeted covering map and $\dim(X)=\dim(X/\tau)$. 

	In this case, the induced map gives an isomorphism 
	\[q^*\from \cohomology{2k}{X/\tau}{\Q}\to \cohomology{2k}{X}{\Q}^\tau\leteq \{\alpha\in \cohomology{2k}{X}{\Q}\from 
	\forall m\in M \quad\tau(m)^*(\alpha)=\alpha\}\]
	of the rational cohomology groups for any $k\in \N$.
\end{thm}
\begin{proof}
	Endow $M$ with the discrete topology. 
	Then $M$ is a compact Lie group, and $\tau$ gives a smooth action. 
	Consequently, $\tau$ is proper by \cite[Corollary 21.6.]{lee2012smoothmanifolds}, so  \cite[Theorem 21.13]{lee2012smoothmanifolds} gives the statement on the smooth structure and the covering map. 
	All stabilisers of free actions are trivial, hence every orbit has cardinality $|M|$ by the orbit–stabilizer theorem, hence the number of sheets is uniformly $|M|$. 
	For the dimension, note that $\dim_\R(X/\tau) = \dim_\R(X)-\dim_\R(M) = \dim_\R(X)$ by \cite[Theorem 21.10]{lee2012smoothmanifolds}.
	The complex case follows from \cite[Corollary~A.7]{ComplexAbelianVarieties}.
	
	Finally, \cite[Proposition~3G.1]{HatcherAT} is applicable by above and gives the statement about the cohomology.
\end{proof}

We are ready to introduce the main notion of this subsection.
\begin{defn}[Holomorphic line bundle associated to isotropic sublattice data] \label{defn:lineBundleToD}
	Let $\mathfrak{D}=(h,V,V_\Re,L_\Re,\Lambda_\Re, \Gamma)$ be an isotropic sublattice data. 
	The restriction $\tau\leteq \tau_\mathfrak{D}|_\Lambda \from \Lambda \to \Biholo(V)$ gives a free group action of $\Lambda$. 
	\eqref{eq:rhoProperty} shows that 	$\rho\leteq \rho_\mathfrak{D}|_\Lambda \from \Lambda \to \Biholo(V\times \C)$ is a free group action since $\Im h(\Lambda,\Lambda)\subseteq \Im h(L,\Lambda)\subseteq \Z$.  
	Now the projection $\trivialLineBundle{V}\from V\times \C\to V$ to the first factor (the trivial line bundle) is a $\Lambda$-equivariant map by \autoref{defn:mapsInducedbyD}, i.e. 
	for any $\lambda\in \Lambda$, the diagram on the left of \eqref{eq:pD} 	commutes. 
	
	Let $\LD{\mathfrak{D}}\leteq (V\times \C)/\rho$ and 
	$\complexTorus{\mathfrak{D}} \leteq V/\tau = V/\Lambda$ (a complex $\dim_\C(V)$-torus). 
	These are complex manifolds by \autoref{lem:quotientManifoldCohomology} and  $\trivialLineBundle{V}$ descends to a holomorphic map $\pD{\mathfrak{D}}\from \LD{\mathfrak{D}}\to \complexTorus{\mathfrak{D}}$. 
	This $\pD{\mathfrak{D}}$ is the \emph{holomorphic line bundle associated to $\mathfrak{D}$}.
	\begin{equation}\label{eq:pD}
		\begin{tikzcd}
		V\times\C \ar[r,"\rho(\lambda)"] \ar[d,"\trivialLineBundle{V}"] & 
		V\times\C \ar[d,"\trivialLineBundle{V}"] 
		\\
		V\ar[r,"\tau(\lambda)"] & 
		V
	\end{tikzcd}
	\qquad
	\rightsquigarrow
	\qquad
	\begin{tikzcd}
		\LD{\mathfrak{D}} \ar[d,"{\pD{\mathfrak{D}}}"] 
		\\
		\complexTorus{\mathfrak{D}}
	\end{tikzcd}
	\end{equation}

\end{defn}

	\subsection{Associated action on a holomorphic line bundle}
\label{sec:diffActionOnLineBundle}

\begin{summary}
	We introduce the notion of action of a short exact sequence of groups on fibre bundles. This captures two compatible actions: one on the base space and one on the fibres.
	Given an arbitrary isotropic sublattice data, we construct an action of the \centralByAbelian{} extension from \autoref{sec:parametrisationHermitianForm} on the line bundle from \autoref{sec:associatedLineBundle}. 
	We study the uniformisability properties of this action.

	The idea is to modify the lattice in the construction of \autoref{sec:associatedLineBundle} in two different equivariant ways. We replace this lattice with two larger ones, each of which corresponds to one half of the generating set of the abelianisation of the Heisenberg group. 
	The join of these lattices does not give an action on the trivial bundle and this obstacle makes it possible to obtain the non-trivial commutators of the Heisenberg group after passing to a suitable quotient bundle.
\end{summary}

\begin{defn}\label{defn:Autp}
	Let $p\from E\to X$ a smooth (respectively holomorphic) fibre bundle.
	$\Diff_p(E)$ (resp. $\Biholo_p(E)$) is the group of $C^\infty$\nobreakdash-diffeomorphisms (resp. biholomorphisms) $\phi$ of $E$ that preserve the fibres, i.e. for which there exists a (unique) $C^\infty$-diffeomorphism (resp. biholomorphism) $p_*\phi$ of $X$ such that $p\circ \phi = (p_*\phi)\circ p$. 
	\[\begin{tikzcd}
		E \ar[d,"p"] \ar[r,"\phi"] &  E \ar[d,"p"] \\	X \ar[r, "p_*\phi"] & X
	\end{tikzcd}\]
	In this case, we say that $\phi$ covers $p_*\phi$. 
	If $p$ is a smooth  $\C$-vector bundle, we also require that $\phi$ restricts to each fibre as a $\C$-linear map. 
	Let $\Diff_p(X)$ (resp. $\Biholo_p(X)$) be the group of $C^\infty$-diffeomorphisms (resp. biholomorphisms) $\beta$ of $X$ such that $\beta$ is covered by some $\phi\in\Diff_p(E)$. 
	Define $\Diff_p^\id(E)$ (resp. $\Biholo_p^\id(E)$) as $\ker(p_*)$, the collection of maps which keep every fibre fixed.
	These groups sit in the following short exact sequence
	\[\begin{tikzcd}[left label]
		\AutSESGeneral{p} \ar[:] & 
		1\ar[r] & 
		\Aut_p^\id(E) \ar[r,inclusion] & 
		\Aut_p(E) \ar[r,"p_*"] & 
		\Aut_p(X) \ar[r] &
		1
	\end{tikzcd}\]
	where $\Aut$ denotes $\Diff$ or $\Biholo$ in the respective cases.
\end{defn}

\begin{defn}\label{defn:actionOfExtensions}
	Let $\epsilon: 1\to C\to G\to M\to 1$ be a short exact sequence of groups, 
	and $p\from L\to X$ be a fibre bundle as in \autoref{defn:Autp}. 
	An \emph{action $\alpha\from \epsilon\acts p$ of $\epsilon$ on $p$} is a morphism $\alpha\leteq (\sigma,\rho,\tau)\from \epsilon\to\AutSESGeneral{p}$ of short exact sequences.
	The action is \emph{faithful}, if the underlying $\epsilon\to\AutSESGeneral{p}$ is a monomorphism, i.e. when all $\sigma,\rho,\tau$ are monomorphisms.
	\begin{equation}\label{diag:exntesionActsOnBundle}
		\begin{tikzcd}[label]
			\epsilon \ar[:] \ar[d,"\alpha"] & 
			1\ar[r] & 
			C \ar[r,"\iota"] \ar[d,"\sigma"]& 
			G \ar[r,"\pi"] \ar[d,"\rho"] & 
			M \ar[r] \ar[d,"\tau"] &
			1
			\\
			\AutSESGeneral{p} \ar[:] & 
			1\ar[r] & 
			\Aut_p^\id(E) \ar[r,inclusion] & 
			\Aut_p(E) \ar[r,"p_*"] & 
			\Aut_p(X) \ar[r] &
			1
		\end{tikzcd}
	\end{equation}
\end{defn}

\begin{rem}\label{rem:actionImpliesCentralCyclicCentre}
	That is, we can consider the action $\rho$ as being decomposed into two actions: $\sigma$ of $C$ on the fibres, and $\tau$ of $M$ on the base space.
	Suppose that the action is faithful, $G$ is a finite group, $p\from L\to X$ is a line bundle and $X$ is connected.
	Then each $\sigma(c)$ acts as multiplication by some $|C|$th root of unity on all fibres. 
	The smoothness of $\sigma(c)$ and the connectivity of $X$ force this root to be the same for all fibres (for a fixed $c$). 
	Hence, on one hand, $C$ injects to $\C^\times$, thus $C$ is in fact a cyclic group.
	On the other hand, $\sigma(c)$ commutes with  $\rho(g)$ for all $c\in C$ and $g\in G$. 
	Thus the injectivity of these maps imply that the image of $C$ lies in $\Center(G)$. 
	In fact, this observation motivated the study of \centralByAbelian{} extensions, cf. \autoref{defn:centralByAbelian} and \cite{Heisenberg}.
\end{rem}

We need some technical notions to define a subclass of actions that will be suitable to undergo a uniformisation process in \autoref{sec:uniformisation}.

\begin{defn}[Cohomology]
	We denote by $\cohomology{\bullet}{X}{R}\leteq \bigoplus_{k=0}^\infty \cohomology{k}{X}{R}$ the (singular) \emph{cohomology} with coefficients from $R$ with $\cuptimes$-product as multiplication \cite[\S3]{HatcherAT}. 
	Write 
	$\cohomology{2\bullet}{X}{R}\leteq \bigoplus_{k=0}^\infty \cohomology{2k}{X}{R}$.
\end{defn}

\begin{defn}[K-theory]
	Let $X$ be a compact Hausdorff space. 
	First, define $K^0(X)$ as the Grothendieck completion of the monoid given by the set of isomorphism classes of vector bundles over $X$ under Whitney sums \cite[Definition 2.1.1]{park2008complex}. 
	Next, define $K^{-1}(X)\leteq 
	\widetilde{K}^0((X\times \R)^+) \leteq 
	\ker(K^0((X\times \R)^+)\to K^0(\{\infty\}))	
	$,
	the reduced $K^0$ of the one-point compactification $(X\times \R)^+\leteq (X\times \R)\cup\{\infty\}$ of $X\times \R$
	\cite[\S2.5-\S2.6]{park2008complex}).
	Finally, write $K^\bullet(X)\leteq K^0(X)\oplus K^{-1}(X)$.
\end{defn}

\begin{defn}\label{defn:uniformisableAction}
	We call the action $\alpha$ from \autoref{defn:actionOfExtensions} \emph{uniformisable} if all of the following hold: 
	\begin{itemize}
		\item $M$ is finite, 
		\item $X$ is a smooth compact manifold, 
		\item $K^0(X)$ is a free $\Z$-module, 
		\item $p\from L\to X$ is a line bundle,  
		\item $\tau\from M\to \Diff(X)$ is a free group action, and 
		\item $M$ acts trivially on the cohomologies through the composition  \[M\xrightarrow{\tau}\Diff(X)\xrightarrow{(-)^*} \Aut(\cohomology{2\bullet}{X}{\Q})\] where $(-)^*\from f\mapsto f^*$ with $f^*$ being the induced map on the cohomology.
	\end{itemize}
	
\end{defn}

\begin{lem}[Künneth formula K-theory, {\cite[Proposition 3.3.15]{park2008complex}}]
	\label{lem:Kunneth}
	If $X,Y$ are compact smooth manifolds such that $K^\bullet(X)$ is a free $\Z$-module, 
	then $K^\bullet(X\times Y)\isom K^\bullet(X)\tensor_\Z K^\bullet(Y)$.
\end{lem}

\begin{exmp}\label{exmp:Ktorus}
	For the $1$-sphere, $K^\bullet(\complexCircle)\isom \Z^2$, see \cite[Example 2.8.1]{park2008complex}. 
	Thus for the $n$-torus $\torus{n}\leteq \complexCircle\times \dots \times \complexCircle$, \autoref{lem:Kunneth} gives 
	$K^\bullet(\torus{n})\isom \Z^{2^n}$ by induction.
\end{exmp}

\begin{prop}\label{lem:HeisenbergActionOnLineBudnles}
	Let $\mathfrak{D}=(h,V,V_\Re,L_\Re,\Lambda_\Re, \Gamma)$ be an isotropic sublattice data. 
	Let $\mu_\mathfrak{D}$ be from \autoref{lem:quotientHermitianForm} and 
	$\pD{\mathfrak{D}}$, ${\LD{\mathfrak{D}}}$ and $\complexTorus{\mathfrak{D}}$ be from \autoref{defn:lineBundleToD}.
	
	Then there exists a faithful uniformisable action 
	\begin{equation}\label{diag:HeisenbergActionOnLineBundle}
		\begin{tikzcd}[label, column sep = 3mm]
			\HeisenbergFunctor(\mu_\mathfrak{D}) \ar[:] \ar[d,"\alpha_\mathfrak{D}"]& 
			1\ar[r] & 
			\Gamma/\Z \ar[r,"{\HeisenbergMono[\mu_\mathfrak{D}]}"] \ar[d,dashed,mono,"\exists\sigma"]& 
			\HH(\mu_\mathfrak{D}) \ar[r,"{\HeisenbergEpi[\mu_\mathfrak{D}]}"] \ar[d,dashed,mono,"\exists\rho"]&
			(L_\Re/\Lambda_\Re)\times (L_\Im/\Lambda_\Im) \ar[r] \ar[d,dashed,mono,"\exists\tau"]& 
			1
			\\
			\Biholo_{\pD{\mathfrak{D}}} \ar[:] & 
			1 \ar[r] & 
			\Biholo_{\pD{\mathfrak{D}}}^\id({\LD{\mathfrak{D}}}) \ar[r,inclusion] & 
			\Biholo_{\pD{\mathfrak{D}}}({\LD{\mathfrak{D}}}) \ar[r,"p_*"] & 
			\Biholo(\complexTorus{\mathfrak{D}}) \ar[r]&
			1 
		\end{tikzcd}.
	\end{equation}
\end{prop}

\begin{proof}
	Consider \[\sigma_\mathfrak{D} \from \Gamma\to \Biholo(V\times \C),\qquad 
	\rho_\mathfrak{D} \from  L\to \Biholo(V\times \C),\qquad
	\tau_\mathfrak{D}\from L\to \Biholo(V)\] 
	from \autoref{defn:mapsInducedbyD}. 
	While $\sigma_\mathfrak{D}$ and $\tau_\mathfrak{D}$ are group morphisms, $\rho_\mathfrak{D}$ is not in general, but certain restrictions are. 
	Indeed, since  $\sigma_\mathfrak{D}|_\Z=\id$ by definition and 
	$\Im h (L_\Re, \Lambda_\Im)=\Im h(\Lambda_\Re, L_\Im)=\Im h(L,\Lambda)\subseteq \Z$ by \autoref{defn:isotropicSublatticeData}, 
	\eqref{eq:rhoProperty} shows that, after all,  $\rho_\mathfrak{D}$ restricted to the lattices $L_\Re\oplus \Lambda_\Im$ and to $\Lambda_\Re\oplus L_\Im$ are both group morphisms. 
	Since these lattices both contain $\Lambda$ (and are abelian groups), 
	for every $\lambda\in \Lambda$ and every 
	\[(f,g)\in \{(\rho_\mathfrak{D}(l), \tau_\mathfrak{D}(l)):l\in L_\Re\cup L_\Im\}\cup \{(\sigma_\mathfrak{D}(c),\id):c\in \Gamma\}\]
	the following diagram on the left is commutative 
	\begin{equation}
		\label{diag:descendedMapsToLineBundles}
		\begin{tikzcd}[sep=small]
			V\times \C \ar[rr,"f"]\ar[dr,"\rho_\mathfrak{D}(\lambda)"] \ar[dd,"\trivialLineBundle{V}"]&& 
			V\times \C \ar[dr,"\rho_\mathfrak{D}(\lambda)"]  \ar[dd,near start,"\trivialLineBundle{V}"]
			\\
			& V\times \C  \ar[rr,near start,crossing over,"f"]&& 
			V\times \C  \ar[dd,"\trivialLineBundle{V}"]
			\\
			V \ar[dr,"\tau_\mathfrak{D}(\lambda)"'] \ar[rr,near start,"g"]&& 
			V \ar[dr,"\tau_\mathfrak{D}(\lambda)"]
			\\
			& V \ar[from=uu,near start,crossing over,"\trivialLineBundle{V}"] \ar[rr,"g"] &&
			V
		\end{tikzcd}
		\qquad
		\rightsquigarrow
		\qquad
		\begin{tikzcd}
			{\pD{\mathfrak{D}}} \ar[r,"{[f]}"] \ar[d,"p"]& 
			{\pD{\mathfrak{D}}} \ar[d,"p"]
			\\
			\complexTorus{\mathfrak{D}} \ar[r,"{[g]}"]& 
			\complexTorus{\mathfrak{D}}
		\end{tikzcd}
	\end{equation}
	where $\trivialLineBundle{V}$ is the trivial line bundle as in \autoref{defn:lineBundleToD}. 
	Hence these $\Lambda$-equivariant morphisms descend to maps between the quotient spaces $\LD{\mathfrak{D}}=(V\times \C)/(\rho_\mathfrak{D}|_\Lambda)$ and $\complexTorus{\mathfrak{D}}=V/(\tau_\mathfrak{D}|_\Lambda)$ as indicated in the right side of the diagram \eqref{diag:descendedMapsToLineBundles} above  where  $\pD{\mathfrak{D}}\from {\pD{\mathfrak{D}}}\to \complexTorus{\mathfrak{D}}$ is the holomorphic line bundle corresponding to $\mathfrak{D}$.

	Being motivated by the proof of \cite[Proposition~5.3]{Heisenberg}, we can now define the maps forming the action \eqref{diag:HeisenbergActionOnLineBundle}. 
	\begin{equation}\label{eq:HeisenbergActionDefinition}
		\begin{aligned}
		\sigma\from 
		c+\Z&\mapsto [\sigma_\mathfrak{D}(c)]
		\\
		\rho\from 
		(l_\Re+\Lambda_\Re,l_\Im+\Lambda_\Im,c+\Z)&\mapsto [\rho_\mathfrak{D}(l_\Im)]\circ \sigma(c)\circ [\rho_\mathfrak{D}(l_\Re)]
		\\
		\tau\from 
		(l_\Re+\Lambda_\Re,l_\Im+\Lambda_\Im)&\mapsto [\tau_\mathfrak{D}(l_\Re+l_\Im)]
	\end{aligned}\end{equation}
	We check that these maps are indeed group monomorphisms and make \eqref{diag:HeisenbergActionOnLineBundle} commute.
	First, consider $\sigma$. 
	Since $\sigma_\mathfrak{D}|_{\Z}=\id$, the map $\sigma$ does not depend on the choice of the representative. 
	By definition, $\sigma$ is a group morphism.
	Note that if $\sigma(c+\Z)=\id$, then $c\in \ker(\sigma_\mathfrak{D})=\Z$, so $\sigma$ is indeed injective.

	Consider $\rho$. 
	If $\lambda\in \Lambda$, then $\rho_\mathfrak{D}(\lambda)$ maps the orbits of $\rho_\mathfrak{D}|_\Lambda$ to themselves (because $\rho_\mathfrak{D}|_\Lambda$ is a group morphism), hence the induced map  $[\hat\rho(\lambda_j)]\in\Biholo({\pD{\mathfrak{D}}})$ is trivial, thus $\rho$ is well-defined as a map of sets.
	To check that it actually is a group morphism, first note that 
	$\rho_\mathfrak{D}(l)\circ \sigma_\mathfrak{D}(c) = \sigma_\mathfrak{D}(c)\circ \rho_\mathfrak{D}(l)$ by definition, so this commutativity relation descends to $\rho$ and $\sigma$. 
	Second, 
	$\rho_\mathfrak{D}(l_\Im')\circ \rho_\mathfrak{D}(l_\Re) 
	=\rho_\mathfrak{D}(l_\Im'+l_\Re) 
	=\rho_\mathfrak{D}(l_\Re+l_\Im') 
	=\rho_\mathfrak{D}(l_\Re)\circ \rho_\mathfrak{D}(l_\Im')\circ \sigma_\mathfrak{D}(-\Im h(l_\Re,l_\Im'))$
	from \eqref{eq:rhoProperty}. 
	So for arbitrary elements $g=(l_\Re+\Lambda_\Re,l_\Im+\Lambda_\Im,c+\Z)$ and $g'=(l_\Re'+\Lambda_\Re,l_\Im'+\Lambda_\Im,c'+\Z)$ of $\HH(\mu_\mathfrak{D})$, we get
	\begin{align*}
		\rho(g*g')&=
		\rho((l_\Re+l_\Re'+\Lambda_\Re, l_\Im+l_\Im'+\Lambda_\Im, c+\Im h(l_\Re,l_\Im')+c'+\Z))
		\\&=
		[\rho_\mathfrak{D}(l_\Im+l_\Im')]\circ \sigma(c+\Im h(l_\Re,l_\Im')+c') \circ	[\rho_\mathfrak{D}(l_\Re+l_\Re')]
		\\&=
		[\rho_\mathfrak{D}(l_\Im)]\circ [\rho_\mathfrak{D}(l_\Im')]\circ 
		\sigma(c)\circ \sigma(\Im h(l_\Re,l_\Im'))\circ \sigma(c') \circ	
		[\rho_\mathfrak{D}(l_\Re)]\circ [\rho_\mathfrak{D}(l_\Re')]
		\\&=
		([\rho_\mathfrak{D}(l_\Im)]\circ \sigma(c)\circ  [\rho_\mathfrak{D}(l_\Re)] )\circ
		([\rho_\mathfrak{D}(l_\Im')]	\circ	\sigma(c') \circ [\rho_\mathfrak{D}(l_\Re')])
		\\&=
		\rho(g)\circ \rho(g')
	\end{align*}
	by using \autoref{rem:Hdef}, hence $\rho$ is a group morphism. 

	By the very definition of $\complexTorus{\mathfrak{D}}=V/(\tau_\mathfrak{D}|_{\Lambda})$, the map $\tau$ is well defined and it is injective. It is also morphism since $\tau_\mathfrak{D}$ is. 

	Note that $\rho\circ \HeisenbergMono[\mu_\mathfrak{D}] = \sigma$ follows from \eqref{eq:HeisenbergActionDefinition} as $[\rho_\mathfrak{D}(0)]=\id$, 
	and $\tau\circ \HeisenbergEpi[\mu_\mathfrak{D}] = ({\pD{\mathfrak{D}}})_*\circ \rho$ is a consequence of the commutativity of \eqref{diag:descendedMapsToLineBundles}. Hence the diagram \eqref{diag:HeisenbergActionOnLineBundle} from the statement is indeed commutative. 
	Finally, the 4-lemma implies the injectivity of $\rho$, hence that of $\alpha_\mathfrak{D}$. 
	
	We check that $\alpha_\mathfrak{D}$ is uniformisable. 
	Indeed, $(L_\Re/\Lambda_\Re)\times (L_\Im/\Lambda_\Im)\isom L/\Lambda$ is finite as it is a quotient of free abelian groups of the same rank.
	\autoref{defn:lineBundleToD} shows that $p$ is a holomorphic (hence smooth) line bundle. 
	$\complexTorus{\mathfrak{D}}$ is diffeomorphic to $\dim_\R(V)$-torus, so $K^\bullet(\complexTorus{\mathfrak{D}})$ is free by \autoref{exmp:Ktorus}.
	The action $\tau$ is free as it is given by translation by elements of the group $\Lambda$. 
	For the final part, note that $\tau_\mathfrak{D}(l)$ is homotopic to the identity on $V$ for every $l\in L$, hence so is $\tau(l)$ on $\complexTorus{\mathfrak{D}}$, thus the induced map on the (rational) cohomologies is trivial.
\end{proof}

	\subsection{Uniformisation using number theory and K-theory}
\label{sec:uniformisation}

\begin{summary}
	We start from a uniformisable action of a short exact sequence of groups on a line bundle. (For example, for every finite Heisenberg group with cyclic centre, we have one such action but on different line bundles as constructed in \autoref{sec:diffActionOnLineBundle}). 
	We construct a direct complement of the line bundle on which the exact sequence acts. 
	We want the resulting trivial bundle to depend only on the (dimension) of the base space, and not on the line bundle nor on the exact sequence we started from. (With this, every Heisenberg group of bounded rank would act simultaneously on this single space.)
	
	To obtain bounds on the rank vector bundle, we cannot use the usual general compactness arguments. Instead, we generalise \cite{Riera} and \cite{specialGroups} to construct the complement with the action in 3 steps. 
	
	First, we extend the action by hand from the line bundle to a vector bundle (of controlled rank) whose Chern character lies in a coset of a specific ideal of the cohomology ring.
	
	Second, we use the isomorphism of the K-theory and rational cohomology to extend this action even more to a vector bundle whose Chern character is trivial (except in degree $0$). The rank of the next bundle still depends only on (the dimension of) the base space.
	
	Finally, we conclude using K-theory that the resulting vector bundle is necessarily isomorphic to the trivial one
	\end{summary}

Given a fixed group extension $G$ of $M$ and 
and a fixed action $\tau$ of $M$ on a space $X$,  
we are interested in finding many vector bundles $E$ over $X$ together with actions of $G$ that are compatible with $\tau$. 
The next definition captures all such possibilities formally. 

\begin{defn}[$\vectorBundleEquivariant{\epsilon}{\tau}{X}$]
\label{defn:VvectorBundleActions}
	Let $\epsilon$ be a short exact sequence of groups and  
	$\tau$ be a group action on a smooth manifold $X$ as indicated in the following diagram.
	\begin{equation}\label{eq:epsTau}
		\begin{tikzcd}[label]
			\epsilon \ar[:] & 
			1\ar[r] & 
			C \ar[r,"\iota"]  & 
			G \ar[r,"\pi"]  & 
			M \ar[r] \ar[d,"\tau"] &
			1
			\\
			 & 
			 & 
			 & 
			 & 
			\Diff(X) &
		\end{tikzcd}
	\end{equation}
	Define 
	$\vectorBundleEquivariant{\epsilon}{\tau}{X}$ to be the set of all diagrams $\alpha$ of the form  \eqref{diag:exntesionActsOnBundle} extending \eqref{eq:epsTau}, 
	i.e. the first rows are identical, 
	and the rightmost vertical maps are the same (up to possibly having different codomains).  
	So may think to elements of $\vectorBundleEquivariant{\epsilon}{\tau}{X}$ as various vector bundles $p\from E\to X$ over the given space $X$ together with an action $\alpha=(\sigma,\rho,\tau)\from \epsilon\acts p$ extending the given $\tau$. 
	It will turn out to be convenient to extend the usual operations and maps from the underlying vector bundles to elements of $\vectorBundleEquivariant{\epsilon}{\tau}{X}$ as follows.
	\begin{itemize}
		\item Define the \emph{Chern character} $\ch\from \vectorBundleEquivariant{\epsilon}{\tau}{X}\to \cohomology{2\bullet}{X}{\Q}$ on actions by $\alpha\mapsto \ch(p)$, the Chern character of the underlying vector bundle $p$ \cite[\S4.1]{HatcherVB}. 
		\item Similarly, define the \emph{rank} $\rank\from \vectorBundleEquivariant{\epsilon}{\tau}{X}\to \N$ by $\alpha\mapsto \rank(p)$, the rank of the underlying vector bundle $p$, which we sometimes will 
		identify with $\vectorBundleEquivariant{\epsilon}{\tau}{X}\xrightarrow{\ch} \cohomology{2\bullet}{X}{\Q}\to \cohomology{0}{X}{\Q}$.
		
		\item The Whitney sum, tensor product and the dual operations on bundles induce 
		\begin{align*}
			\oplus\from \vectorBundleEquivariant{\epsilon}{\tau}{X}\times \vectorBundleEquivariant{\epsilon}{\tau}{X}&\to \vectorBundleEquivariant{\epsilon}{\tau}{X},
			\\
			\alpha_1\oplus\alpha_2&\leteq (\sigma_1\oplus\sigma_2,\rho_1\oplus\rho_2,\tau)\from \epsilon \acts(p_1\oplus p_2)
			\\
			\tensor\from \vectorBundleEquivariant{\epsilon}{\tau}{X}\times \vectorBundleEquivariant{\epsilon}{\tau}{X}&\to \vectorBundleEquivariant{\epsilon}{\tau}{X}, 
			\\
			\alpha_1\tensor\alpha_2&\leteq (\sigma_1\tensor\sigma_2,\rho_1\tensor\rho_2,\tau)\from \epsilon \acts(p_1\tensor p_2)
			\\
			(-)^*\from \vectorBundleEquivariant{\epsilon}{\tau}{X}&\to \vectorBundleEquivariant{\epsilon}{\tau}{X}, 
			\\
			\alpha^*&\leteq (\sigma^*,\rho^*,\tau)\from \epsilon \acts p^*
		\end{align*}
		which are defined fibrewise where the actions $\alpha_i=(\sigma_i,\rho_i,\tau)\from \epsilon_i \acts p_i$ and $\alpha=(\sigma,\rho,\tau)\from \epsilon \acts p$ are all from $\vectorBundleEquivariant{\epsilon}{\tau}{X}$. 
		
		\item Let $\trivialLineBundle{X}\from X\times \C\to X$ be the trivial line bundle over $X$.
		Then $\actionOnTrivialLineBundle{\epsilon}{\tau} \leteq (1,\rho,\tau)\from \epsilon\acts \trivialLineBundle{X}$ is an action from $\vectorBundleEquivariant{\epsilon}{\tau}{X}$ where the group action $\rho\from G\to \Aut_p(\trivialLineBundle{X})$ is given by $g\mapsto((x,z)\mapsto (\tau(\pi(g))(x),z))$.
		
		\item As usual, for $n \in\Np$, set $\alpha^{\tensor n}=\alpha\tensor \dots\tensor \alpha$, the $n$-fold tensor power; 
		set $\alpha^{\tensor -n}\leteq (\alpha^*)^{\tensor n}$; 
		and $\alpha^{\tensor 0}\leteq \actionOnTrivialLineBundle{\epsilon}{\tau}$. 	
	\end{itemize}
	
\end{defn}

\begin{rem}\label{rem:chRingMorphism}
	$\tensor$ is distributive over $\oplus$ (up to isomorphism) and these operations are compatible with the ring structure of $\cohomology{2\bullet}{X}{\Q}$ via $\ch$: 
	$\ch(\alpha_1\oplus\alpha_2)=\ch(\alpha_1)+\ch(\alpha_2)$, 
	$\ch(\alpha_1\tensor\alpha_2)=\ch(\alpha_1)\cuptimes \ch(\alpha_2)$. 
	(In fact, passing to the Grothendieck completion of the isomorphism classes would give a ring structure on $\vectorBundleEquivariant{\epsilon}{\tau}{X}$.) 
	$\alpha^{\tensor n}\tensor \alpha^{\tensor k}\isom \alpha^{\tensor (n+k)}$ for any action $\alpha$ on a line bundle and $n,k\in\Z$.
	Note that $\actionOnTrivialLineBundle{\epsilon}{\tau}\tensor \alpha\isom \alpha$ for every action  $\alpha\in \vectorBundleEquivariant{\epsilon}{\tau}{X}$.
\end{rem}

\subsubsection{Step 1: Reaching a principal ideal in cohomology using number theory}
\label{sec:uniformisationCohomology}

\begin{summary}
	We show that actions on line bundles can be extended to vector bundles 
	whose positive degree Chern characters are all multiples of a given integer. 
	The rank of the resulting vector bundle does not depend on the integer, but only on the (dimension) of the base space. 
	We look for this vector bundle as a direct sum of certain tensor powers of the given line bundle. 
	Some computation reduces the existence of suitable exponents to a purely number theoretical result which is a simple consequence of the classical modular Waring problem.
\end{summary}

\begin{lem}[Modular Waring problem, {\cite[p.~186,~Theorem~12]{Hardy1922}.}] \label{lem:moduloWaring}
	For every $k,n\in \Np$, there is $\gamma(n,k)\leq 4k$ such that modulo $n$, every integer can be written as a sum of at most $\gamma(n,k)$ many $k$th powers.
\end{lem}

\begin{lem}\label{lem:numberTheory}
	There exists $R_1\from \N^2\to \N$ such that 
	for every 
	$n\in\Np$, every finite multiset $S$ of integers and every $\delta\in\Np$, 
	there exists a multiset $T \supseteq S$ of integers of total cardinality at most $R_1(n,|S|)$ satisfying $\sum_{t\in T} t^k \equiv 0 \pmod\delta{}$ for every $1\leq k\leq n$.
\end{lem}
\begin{proof}
	If $n=0$ or $S=\emptyset$, then $T=\emptyset$ satisfies the statement. 
	Otherwise, for $1\leq k\leq n$, let $W_{k}$ be the smallest positive integer $N$ with the property that modulo any natural number, $-1$ can be expressed as a sum of at most $N$ many $k$th powers. 
	This exists by \autoref{lem:moduloWaring} and $W_{k}\leq 4k$. We show that $R_1(n,m)\leteq (m+1)\prod_{k=2}^n (W_{k}+1)$ satisfies the statement. 
	
	Denote $p_k(T)\leteq \sum_{t\in T}t^k$ for any multiset $T$. 
	Let $T_1\leteq S+\{-p_1(S)\}$, the multiset obtained from $S$ by increasing the multiplicity of $-p_1(S)$ by $1$.
	For $2\leq k\leq n$, pick $P_k$ of total cardinality at most $W_k$ such that $p_k(P_k)\equiv -1\pmod {\delta}$ and define the multiset $T_k\leteq P_k+\{1\}$. 
	Define the multiset $T\leteq \{\prod_{k=1}^n t_k: t_k\in T_k\text{ for all }1\leq k\leq n\}$.
	By construction $|T|\leq R_1(n,|S|)$, and $S\subseteq T$ because $S\subset T_1$ and $1\in T_k$ for $2\leq k\leq n$.  
	Also the divisibility chain $\delta\divides p_k(T_k)\divides\prod_{i=1}^n p_k(T_i)=p_k(T)$ holds for any $1\leq k\leq n$, so $T$ is as stated.
\end{proof}
\begin{rem}\label{rem:R1}
	The proof actually shows that one can choose $R_1$ to satisfy the upper bound $R_1(n,m)\leq (m+1)\prod_{k=2}^n (4k+1) \leq (m+1) 4^n (n+1)!$. This upper bound is quite possibly very far from the smallest possible value of $R_1$. 
\end{rem}

\begin{lem}\label{lem:cohomologyVanishAfterDimAndFG}
	For a compact manifold $X$, 
	$\cohomology{\bullet}{X}{\Z}$ is a finitely generated $\Z$-module and 
	$\cohomology{k}{X}{\Z}=0$ for $k>\dim_\R(X)$.
\end{lem}
\begin{proof}
	This follows from \cite[Corollaries~A.8-9, Theorems~3.2,3.26]{HatcherAT}.
\end{proof}

\begin{rem}\label{rem:ChFromChernClass}
	The natural embedding $\Z\embeds \Q$ induces a morphism $\cohomology{\bullet}{X}{\Z}\to \cohomology{\bullet}{X}{\Q}$ whose kernel is the set of torsion elements. Denote by $\cohomologyRes{2\bullet}{X}{\Q}{\Z}$ the image of $\cohomology{2\bullet}{X}{\Z}$. 
	
	Apart from the above compatibility of the Chern character $\ch$ with $\oplus$ and $\tensor$, we will need to compute it on line bundles explicitly. 
	If $p\from L\to X$ is a line bundle over a compact manifold $X$, then 
	$\ch(p) = \sum_{k=0}^\infty \frac{1}{k!} c_1(p)^{\cuptimes k}\eqlet \exp(p)$, a finite sum from \autoref{lem:cohomologyVanishAfterDimAndFG} \cite[\S4.1. The Chern Character]{HatcherVB}. 
	Here $c_1(p)\in \cohomologyRes{k}{X}{\Q}{\Z}$ is the first Chern class of $p$, see 
	\cite[\S3.1. Stiefel-Whitney and Chern Classes]{HatcherVB}.
\end{rem}

\begin{lem}\label{lem:chIndHZ}
	There is $R_2\from \N\to \N$ with the following property.
	If $X$ is a smooth compact manifold, $\alpha\in\vectorBundleEquivariant{\epsilon}{\tau}{X}$ is an action 
	of rank $1$, and $d\in\Np$, 
	then there is an action $\alpha[d] \in \vectorBundleEquivariant{\epsilon}{\tau}{X}$ 
	of rank at most $R_2(\dim_\R (X))$ 
	such that 
	\[\ch(\alpha\oplus\alpha[d])-\rank(\alpha\oplus\alpha[d]) \in d\cdot \cohomologyRes{2\bullet}{X}{\Q}{\Z}.\]
\end{lem}

\begin{proof}
	We show that $R_2(m)\leteq R_1(\lfloor m/2\rfloor, 1)-1$ satisfies the statement where $R_1$ is from \autoref{lem:numberTheory}.
	Let $T$ be the multiset provided by \autoref{lem:numberTheory} when applied with $n\leteq \lfloor \dim_\R (X)/2\rfloor$, $S\leteq \{1\}$ and $\delta\leteq n!d$. 
	We show that 
	\begin{equation}\label{eq:alpha(d)}
		\alpha[d]\leteq \bigoplus_{\mathclap{t\in T-\{1\}}} \alpha^{\tensor t}
	\end{equation}satisfies the statement. 
	Then as $\rank(\alpha)=1$,  one has $\rank(\alpha[d])=\sum_{t\in T-\{1\}} \rank(\alpha)^t=|T|-1\leq R_1(n,2)-1=R_2(\dim (X))$.
	As $\alpha$ acts on a line bundle $p$, $\ch(\alpha)=\exp(c_1(\alpha))$ where $c_1(\alpha)\in \cohomologyRes{2}{X}{\Q}{\Z}$ is the first Chern class of the line bundle $p$, cf. \autoref{rem:ChFromChernClass}. 
	Then using \autoref{rem:chRingMorphism} and that $\cohomologyRes{2k}{X}{\Q}{\Z}=0$ for $k>n$ by \autoref{lem:cohomologyVanishAfterDimAndFG}, we obtain 
	\begin{align*}
		\ch(\alpha\oplus\alpha[d]) &= 
		\ch(\bigoplus_{t\in T} \alpha^{\tensor t}) = 
		\sum_{t\in T} \ch(\alpha)^{\cuptimes t} = 
		\sum_{t\in T} \exp(t c_1(\alpha)) = 
		\sum_{t\in T} \sum_{k=0}^\infty \frac{t^k}{k!} c_1(\alpha)^k \\&= 
		\rank(\alpha\oplus\alpha[d]) + \sum_{k=1}^n\sum_{t\in T}  \frac{t^k}{k!} c_1(\alpha)^k \in 
		\rank(\alpha\oplus\alpha[d]) + d\cdot \cohomologyRes{2\bullet}{X}{\Q}{\Z},
	\end{align*}
	since $\sum_{t\in T}t^k\in n!d\Z$ for every $1\leq k\leq n$ by \autoref{lem:numberTheory}.
\end{proof}

\begin{rem}\label{rem:R2}
	The proof shows that $R_2$ can be chosen to satisfy the inequality $R_2(m) \leq R_1(\lfloor \frac{m}{2}\rfloor, 1)$, where $R_1$ is from \autoref{lem:numberTheory}. 
	Considering \autoref{rem:R1}, we get $R_2(m)\leq 2^{m+2} \left(\left\lfloor\frac{m}{2}\right\rfloor+1\right)!$ which is quite possibly very far from the optimum.

	As discussed in the introduction of this section, it is of key importance for the main application that the bound $R_2$ is independent of $\epsilon$, $\alpha$ and $d$, and depends only on (the dimension of) $X$.
\end{rem}

\subsubsection{Step 2: Reaching trivial Chern character using K-theory}
\label{sec:uniformisationTrivialChern}
\begin{summary}
	We show that over compact manifolds, every element of the cohomology ring whose positive degree coefficients are all multiple of a well-chosen integer 
	can be realised as the Chern character of a pullback bundle of rank depending only on (the dimension of) the base space.
	Considering pullback bundles is important, as it automatically endows the resulting bundle with a (typically non-faithful) action. 
	We use classical results from K-theory and that cohomologies beyond the dimension all vanish.
\end{summary}

\begin{lem}[{\cite[Part II, \S9.1, Theorem~1.2, p.112]{Husemoller}.}]\label{lem:chopFromVectorBundle}
	Let $X$ be a finite dimensional CW-complex (e.g. a compact manifold). 
	If $p$ is a complex vector bundle over $X$ of rank at least $\dim (X)/2$,
	then $p\isom p_0\oplus\trivialLineBundle{X}^{k}$ for some vector bundle $p_0$ of rank $\lfloor \dim (X)/2\rfloor$ and $k\in\N$ where $\trivialLineBundle{X}$ is the trivial line bundle over $X$.
\end{lem}

\begin{lem}[{\cite[Propositions~2.1.4-5]{park2008complex}}]\label{lem:elementsOfK0}
	If $X$ is a compact Hausdorff space (e.g. a compact manifold), then every element of $K^0(X)$ can be represented as $[p]-[\trivialLineBundle{X}^{\oplus k}]$ for some vector bundle $p$ whose rank at each connected component is at least $k$.
	
	Furthermore, $[p]-[\trivialLineBundle{X}^{\oplus k}]=0\in K^0(X)$ if and only if there exists $m\in \N$ such that $p\oplus\trivialLineBundle{X}^{\oplus m}\isom \trivialLineBundle{X}^{\oplus (k+m)}$.
\end{lem}

\begin{lem}\label{lem:chIso}
	If $X$ is a compact Hausdorff space, then 
	the \emph{Chern character}
	\[\ch\from K^0(X)\tensor \Q\to \cohomology{2\bullet}{X}{\Q}=\cohomology{2\bullet}{X}{\Z}\tensor \Q\] is a ring  isomorphism.
	
	In particular, if $K^0(X)$ is free, then $\ch\from K^0(X)\to \cohomology{2\bullet}{X}{\Q}$ is injective.
\end{lem}
\begin{proof}
	\cite[Théor\`eme~3]{Karoubi1963-1964} or \cite[Theorem,~Corollary \S2.4, p.19]{AtiyahHirzebruch} gives the statement with singular cohomology replaced by  \v{C}ech cohomology. 
	But these cohomologies coincide for compact manifolds by 
	\cite[\S3.1,~p.201]{HatcherAT},  \cite[Proposition~2.1.7]{park2008complex}, 
	\cite[Corollary A.12]{HatcherAT} and 
	\cite[\S3.3,~p.257]{HatcherAT}. 
\end{proof}

\begin{prop}\label{lem:actionWithGivenCh}
	Let $\epsilon:1\to C\to G\to M\to 1$ be a short exact sequence of groups such that $M$ is finite. 
	Let $X$ be a smooth compact manifold. 
	Let $\tau\from M\to\Diff(X)$ be a free action such that the induced  composition  $M\xrightarrow{\tau}\Diff(X)\xrightarrow{(-)^*} \Aut(\cohomology{2\bullet}{X}{\Q})$ is the identity.
	
	Then there is $d\in\Np$ such that 
	for any $\chi\in d\cdot \cohomologyRes{2\bullet}{X}{\Q}{\Z}$, 
	there is an action $\alpha_\chi\in \vectorBundleEquivariant{\epsilon}{\tau}{X}$ of rank at most $\dim_\R (X)/2$ such that 
	\[\ch(\alpha_\chi)-\chi\in \cohomologyRes{0}{X}{\Q}{\Z}.\]
\end{prop}
\begin{proof}
	Let $Y\leteq X/\tau$ be the quotient manifold from \autoref{lem:quotientManifoldCohomology}, and $q\from X\to Y$ be the natural projection.
	Consider the diagram below.
	\[\begin{tikzcd}
		K^0(X)\tensor \Q  \ar[d,iso',"\ch"] & 
		&
		K^0(Y)\tensor \Q \ar[d,iso',"\ch"] 	\ar[ll,iso',dashed,"q^*"']
		\\
		\cohomology{2\bullet}{X}{\Q} & 
		\cohomology{2\bullet}{X}{\Q}^\tau \ar[l,identity]& 
		\cohomology{2\bullet}{Y}{\Q} \ar[l,iso',"q^*"']
	\end{tikzcd}\]
	The manifold $X$ is compact and so is its quotient $Y$, therefore \autoref{lem:chIso} implies that the vertical maps indicated by $\ch$ are isomorphisms. 
	$\cohomology{2\bullet}{X}{\Q}^\tau = \cohomology{2\bullet}{X}{\Q}$ by assumption, so \autoref{lem:quotientManifoldCohomology} shows that the bottom row of the diagram is also an isomorphism.
	Hence the commutativity of the diagram implies that $q^*\from K^0(Y)\tensor \Q\to K^0(X)\tensor \Q$ is also a ring isomorphism.
	
	Since $X$ is compact, $\cohomology{2\bullet}{X}{\Z}$ is finitely generated by \autoref{lem:cohomologyVanishAfterDimAndFG}, hence so are its quotients, thus we may pick a finite $\Z$-module generating set $B$ of $\cohomologyRes{2\bullet}{X}{\Q}{\Z}$. 
	Then from the above diagram, for every $b\in B$ there exists $\upsilon_b\in K^0(Y)$, $d_\gamma\in \Np$ such that $\ch(q^*(\upsilon_\gamma\tensor \frac{1}{d_\gamma}))=\gamma$. 
	Let $d\leteq \prod_{\gamma\in B} d_\gamma\in \Np$. 
	Then by above $\ch(q^*(\frac{d}{d_\gamma}\upsilon_\gamma)=d\cdot \gamma$, thus 
	$d\cdot \cohomologyRes{2\bullet}{X}{\Q}{Z}\subseteq \ch(q^*(K^0(Y)))$. 
	We show that this $d$ satisfies the statement.
	
	Pick $\chi\in d\cdot \cohomologyRes{2\bullet}{X}{\Q}{\Z}$. Then by above, there is $\xi\in K^0(Y)$ such that $\ch(q^*(\xi))=\chi$. 
	Using \autoref{lem:elementsOfK0}, write $\xi=[p_Y]-[\trivialLineBundle{Y}^{\oplus k}]$ for some vector bundle $p_Y\from E\to Y$ over $Y$, and $k\in\N$. 
	By \autoref{lem:chopFromVectorBundle} and \autoref{rem:chRingMorphism}, we may assume that $\rank(p_Y) \leq \lfloor \dim_\R (Y)/2\rfloor=\lfloor \dim_\R(X)/2\rfloor=n$ since $\dim_R(X)=\dim_\R(Y)$. 
	Set $p\leteq q^*(p_Y)\from q^* E\to X$.
	By construction, $\ch(p) - \chi = \ch(q^*(p_Y)) - \ch(q^*(\xi)) = \ch(q^*(\trivialLineBundle{Y}^{\oplus k})= \ch(\trivialLineBundle{X}^{\oplus k})) \in \cohomology{0}{X}{\Z}$ and $\rank(p)=\rank(p_Y)$.
	
	Finally, we give an action $\epsilon\acts p$. Consider the diagram
	\[
	\begin{tikzcd}[left label]
		\epsilon \ar[:] & 
		1\ar[r] & 
		C \ar[r,"\iota"] \ar[d,"\sigma"]& 
		G \ar[r,"\pi"] \ar[d,"\rho"] & 
		M \ar[r] \ar[d,"\tau"] &
		1
		\\
		\AutSES{p} \ar[:] & 
		1\ar[r] & 
		\Diff_{p}^\id(q^* E) \ar[r,inclusion] & 
		\Diff_{p}(q^* E) \ar[r,"p_*"] & 
		\Diff_{p}(X) \ar[r] &
		1
	\end{tikzcd}\]
	where \[\sigma\from c\mapsto \id_{q^* E},\qquad \rho\from g\mapsto ((x,e)\mapsto (\tau(\pi(g))(x),e)) \]
	for $q^{*}E=\{(x,e)\in X\times E:q(x)=p_Y(e)\}$.
	Since $q(\tau(m)(x))=q(x)$ for any $m\in M$ and $x\in X$ by definition, 
	so if $(x,e)\in q^* E$ and $g\in G$, then 
	$q(\rho(g)(x,e))=
	q(\tau(\pi(g))(x))=
	q(x) = 
	p_Y(e)$.
	This shows that $\rho(g)(x,e)\in q^* E$, i.e. that $\rho$ is well-defined. 
	It is clear from the definition of $\rho$ that it preserves the fibres and the exactness of $\epsilon$ implies the commutativity of the diagram. 
	Hence $\alpha_\chi\leteq (\sigma,\rho,\tau)\in \vectorBundleEquivariant{\epsilon}{\tau}{X}$ has the stated properties. 
	
	Note that the constructed $\alpha_\chi$ is typically not faithful. In fact, $\alpha_\chi$ is faithful if and only if $C$ is trivial. 
\end{proof}

\subsubsection{Step 3: Uniformisation through trivial vector bundles}

\begin{summary}
	Starting from a uniformisable action on a line bundle, we take the direct sum of the constructions of \autoref{sec:uniformisationCohomology} and \autoref{sec:uniformisationTrivialChern}
	to obtain a faithful action on a vector bundle whose Chern character is trivial and whose rank depends only on (the dimension of) the base space. 
	Using K-theory, we conclude that the resulting vector bundle is in fact trivial, hence does not depend on the line bundle we started from, only on its base space.
\end{summary}

\begin{lem}[{\cite[Part II, \S9.1, Theorem~1.5, p.112]{Husemoller}.}]\label{lem:stablyIsomBundles}
	Let $X$ be a finite dimensional CW-complex (e.g. a compact manifold). 
	Let $p_1$ and $p_2$ be complex vector bundles over $X$ of rank at least $\dim (X)/2$ such that $p_1\oplus \trivialLineBundle{X}^{\oplus m}\isom p_2\oplus\trivialLineBundle{X}^{\oplus m}$ for some $m\in \N$. 
	Then $p_1\isom p_2$.
\end{lem}

\begin{lem}[Detecting trivial bundles]\label{lem:ChZeroImpliesTrivial}
	Suppose $X$ is a compact manifold such that $K^0(X)$ is free, and let $p$ be a complex vector bundle over $X$ of rank $r\geq \dim_\R(X)/2$ with $\ch(p)\in\cohomology{0}{X}{\Q}$. 
	Then $p\isom \trivialLineBundle{X}^{\oplus r}$, the trivial bundle of rank $r$.
\end{lem}
\begin{proof}
	The Chern character  $\ch\from K^0(X)\to H^\bullet(X,\Q)$ is injective by assumption and \autoref{lem:chIso}. 
	Then since $\ch([p]-[\trivialLineBundle{X}^{\oplus r}]) = 0\in H^\bullet(X,\Q)$ by assumption, $[p]-[\theta^{\oplus r}]=0\in K(X)$. 
	This means by \autoref{lem:elementsOfK0}, that there is $m\in\N$ such that $p\oplus \trivialLineBundle{X}^{\oplus m}\isom \trivialLineBundle{X}^{\oplus(r+m)}$. 
	Finally, since $r\geq \dim_\R(X)/2$, \autoref{lem:stablyIsomBundles} implies that $p \isom \trivialLineBundle{X}^{\oplus r}$.
\end{proof}

\begin{prop}[Uniformisation]\label{prop:complementActionExists}
	There exists an increasing function $R_3\from \N\to \N$ such that 
	whenever $\alpha\in \vectorBundleEquivariant{\epsilon}{\tau}{X}$ is a uniformisable action, 
	there exists a (complementary) action $\alpha^\perp \in \vectorBundleEquivariant{\epsilon}{\tau}{X}$ such that $\alpha\oplus\alpha^\perp$ is an action on a vector bundle isomorphic to the trivial bundles of rank at most $R_3(\dim_\R(X))$.
	
	In particular, if $\alpha\in \vectorBundleEquivariant{\epsilon}{\tau}{X}$ is faithful, then there is a faithful action $\hat\alpha\in \vectorBundleEquivariant{\epsilon}{\tau}{X}$
	\begin{equation}\label{diag:uniformisationDiff}
		\begin{tikzcd}[label]
		\epsilon\ar[:]\ar[d,"\hat\alpha"] & 
		1\ar[r] & 
		C\ar[r,"\iota"]\ar[d,mono,"\sigma"] &
		G\ar[r,"\pi"]\ar[d,mono,"\rho"] & 
		M\ar[r]\ar[d,mono,"\tau"] &
		1
		\\
		\AutSES{p}\ar[:]&
		1\ar[r]&
		\Diff_p^\id(X\times \C^r) \ar[r] &
		\Diff_p(X\times  \C^r) \ar[r,"p_*"] & 
		\Diff(X)\ar[r] &
		1
	\end{tikzcd}
	\end{equation}
	on the trivial bundle $p\leteq \trivialLineBundle{X}^{\oplus r}$ of rank $r\leteq R_3(\dim_\R(X))$, a space independent of $\hat\alpha$.
\end{prop}
\begin{rem}\label{rem:R3}
	The proof shows that one may choose $R_3$ to satisfy the upper bound $R_3(m)\leq R_2(m)+\lfloor m/2\rfloor+1$ where $R_2$ is from \autoref{lem:chIndHZ}. 
	Then \autoref{rem:R2} give a rough estimate of $R_3(m)\leq 2^{m+3} \left(\left\lfloor\frac{m}{2}\right\rfloor+1\right)!$, which again may be very far from the optimum.
\end{rem}

\begin{proof}
	Let $d\in\Np$ be given by \autoref{lem:actionWithGivenCh}, and the action $\alpha[d]\in \vectorBundleEquivariant{\epsilon}{\tau}{X}$ by \autoref{lem:chIndHZ}. 
	Now 
	$-\chi\leteq \ch(\alpha\oplus \alpha[d])-\rank(\alpha\oplus\alpha[d])\in d\cdot \cohomologyRes{2\bullet}{X}{\Q}{\Z}$, 
	so \autoref{lem:actionWithGivenCh} applies and produces an action $\alpha_\chi\in \vectorBundleEquivariant{\epsilon}{\tau}{X}$ such that 
	$\ch(\alpha_\chi) -\chi \in\cohomologyRes{0}{X}{\Q}{\Z}$.
	We show that \[\alpha^\perp\leteq \alpha[d]\oplus \alpha_\chi\oplus\actionOnTrivialLineBundle{\epsilon}{\tau}^{\oplus k}\] satisfies the statement where 
	$\actionOnTrivialLineBundle{\epsilon}{\tau}\in \vectorBundleEquivariant{\epsilon}{\tau}{X}$ is the trivial action from \autoref{defn:VvectorBundleActions}, and 
	$k\in\N$ is the smallest natural number making $\rank(\alpha^\perp)\geq \dim_\R(X)/2$ the definition above. 
	By construction, $\rank(\alpha\oplus\alpha^\perp)\leq  1+R_2(\dim_\R(X)) + \lfloor\dim_\R(X)/2\rfloor \eqlet R_3(\dim_\R(X))$, 
	and 
	$\ch(\alpha\oplus\alpha^\perp) = 
	\ch(\alpha\oplus \alpha[d])+\ch(\hat\alpha_{\chi}) = -\chi + \chi(\alpha_\chi)\in \cohomologyRes{0}{X}{\Q}{\Z} \subseteq  \cohomology{0}{X}{\Q}$ using \autoref{rem:chRingMorphism}.
	Then \autoref{lem:ChZeroImpliesTrivial} applies and gives the first part of the statement.
	
	For the second part, consider $\hat\alpha\leteq \alpha\oplus\alpha^\perp\oplus\actionOnTrivialLineBundle{\epsilon}{\tau}^{\oplus k}$ for $k\leteq r-\rank(\alpha^\perp)\geq 0$, so that $\rank(\hat\alpha)=r$. 
	Then $\hat\alpha$ is faithful since its component $\alpha$ is faithful. 
	By the above discussion, the underlying vector bundle is isomorphic to the trivial one, so $\hat\alpha$ induces the diagram from the statement.
\end{proof}

	\subsection{Compactification}
\label{sec:compactification}

\begin{summary}
	We recall the standard notions of Stiefel and Grassmann bundles. 
	Then we show that every faithful action of a short exact sequence on a complex vector bundle induce actions on these bundles in a natural way. 
	While the total space of a complex vector bundle is typically not compact, that of the Stiefel and Grassmann bundles are. Hence this section gives a way to obtain an action on a compact space from one on a non-compact space.
\end{summary}

Let $p\from E\to X$ be a smooth complex vector bundle of rank $t$ over a smooth manifold. 
Let $1\leq k\leq t$ be an integer and let $h$ be a Hermitian metric on $p$ \cite[Definition~1.7.7]{park2008complex}. 
We briefly recall the following two standard vector bundles associated to $p$, $k$ and $h$ mainly to fix the notation.

\begin{defn}[Associated Stiefel bundle, cf. {\cite[\S1.1, Associated Fiber Bundles]{HatcherVB}, \cite[\S8.1]{Husemoller}}]
	\label{exmp:StiefelBundle}	
	For the Kronecker delta $\delta_{i,j}$, define  \[\Stiefel_k(E,h)\leteq \{(e_1,\dots,e_k)\in( p^{-1}(x))^k:x\in X,h(e_i,e_j)=\delta_{i,j}\}\subseteq E^k\] equipped with the subspace topology of the product space. This comes with a natural projection \[\Stiefel_k(p,h)\from \Stiefel_k(E,h)\to X\] 
	called the \emph{associated Stiefel bundle} whose fibres are diffeomorphic to compact the Stiefel manifold $\Stiefel_k(\C^t)\isom \U(t)/\U(t-k)$.
\end{defn}

\begin{defn}[Associated Grassmann bundle, {\cite[\S1.1, Associated Fiber Bundles]{HatcherVB}}]
	\label{exmp:GrassmannBundle}	
	Define $\Grassmann_k(E)$ to be the quotient of $\Stiefel_k(E,h)$ identifying tuples from the same fibre (of $\Stiefel_k(p,h)$) if they generate the same linear subspace of the fibre (of $p$). 
	Then the natural projection 
	\[\Grassmann_k(p)\from \Grassmann_k(E)\to X\]
	is 	called the \emph{associated Grassmann bundle} whose fibres are diffeomorphic to the compact Grassmann manifold $\Grassmann_k(\C^t)$. 
\end{defn}

At the end of the day, we want to obtain a group action on a smooth compact manifold. 
For this, the next statement about actions of short exact sequences of groups (cf. \autoref{defn:actionOfExtensions}) is useful. 
\begin{prop}[Compactification]\label{lem:compactification}
	Let $\epsilon$ be a short exact sequence of finite groups. 
	Then every smooth faithful action $\hat\alpha\from \epsilon\acts p$ from $\vectorBundleEquivariant{\epsilon}{\tau}{X}$ of the form  \eqref{diag:exntesionActsOnBundle}  
	induces two natural smooth faithful actions from $\vectorBundleEquivariant{\epsilon}{\tau}{X}$:
	\begin{enumerate}
		\item $\Stiefel_k(\hat\alpha,h)\from \epsilon\acts\Stiefel_k(p,h)$ on the Stiefel bundle 
		for every $1\leq k\leq t$ (for a suitable Hermitian metric $h$), and 
		\item $\Grassmann_k(\hat\alpha\oplus\actionOnTrivialLineBundle{\epsilon}{\tau})\from \epsilon\acts\Grassmann_k(p\oplus \trivialLineBundle{X})$ on the Grassmann bundle
		for every $1\leq k\leq t+1$ 
		(where  $\trivialLineBundle{X}$ is the trivial bundle over $X$ which comes with an action  $\actionOnTrivialLineBundle{\epsilon}{\tau}\in\vectorBundleEquivariant{\epsilon}{\tau}{X}$ defined in \autoref{defn:VvectorBundleActions}).
	\end{enumerate}
\end{prop}
\begin{rem}
	The total space of the resulting fibre bundles is compact in all cases, hence all groups of $\epsilon$ act faithfully via $\C^\infty$-diffeomorphisms on the respective smooth manifold.
\end{rem}
\begin{proof}[Proof of \autoref{lem:compactification}]
	First consider the Stiefel case.  
	Let $h_0$ be a Hermitian metric on $p$ whose existence is given by e.g. \cite[Corollary~1.7.10]{park2008complex}. 
	Define a new Hermitian metric $h$ by setting 
	$h(e,e')\leteq \frac{1}{|G|}\sum_{h\in G} h_0(\rho(h)(e), \rho(h)(e'))$ 
	for $e,e'$ coming from the same fibre. 
	By the usual argument, this is $\rho$-invariant, i.e.
	$h(\rho(g)(e),\rho(g)(e'))=h(e,e')$. 
	Write $p_{\Stiefel}\leteq \Stiefel_k(p,h)$. 
	We claim that 
	\begin{equation}\label{diag:StiefelCompactification}
		\begin{tikzcd}[label,column sep=3mm]
			\epsilon \ar[:]\ar[d,"{\Stiefel_k(\hat\alpha,h)}"] & 
			1\ar[r] & 
			C \ar[r,"\iota"] \ar[d,"\sigma_{\Stiefel}"]& 
			G \ar[r,"\pi"] \ar[d,"\rho_{\Stiefel}"] & 
			M \ar[r] \ar[d,"\tau"] &
			1
			\\
			\AutSES{p_{\Stiefel}} \ar[:] & 
			1\ar[r] & 
			\Diff_{p_{\Stiefel}}^\id(\Stiefel_k(E,h)) \ar[r,inclusion] & 
			\Diff_{p_{\Stiefel}}(\Stiefel_k(E,h)) \ar[r,"(p_{\Stiefel})_*"{yshift=2pt}] & 
			\Diff_{p_{\Stiefel}}(X) \ar[r] &
			1
		\end{tikzcd}
	\end{equation}
	is an action as stated where 
	\begin{align*}
		\sigma_{\Stiefel}&\from c\mapsto ((e_1,\dots,e_k)\mapsto (\sigma(c)(e_1), \dots,\sigma(c)(e_k)))\\ 
		\rho_{\Stiefel}&\from g\mapsto ((e_1,\dots,e_k)\mapsto (\rho(g)(e_1), \dots,\rho(g)(e_k))).
	\end{align*}
	Note that these maps are well defined as 
	$e_i\in p^{-1}(x)$ implies that $\rho(g)(e_i)\in p^{-1}(\tau(\pi(g))(x))$ and furthermore that 
	$h(\rho(g)(e_i),\rho(g)(e_j)) = h(e_i,e_j) = \delta_{i,j}$ (the Kronecker delta)
	and similarly for $\sigma_{\Stiefel}$.
	These are group morphisms making the diagram above commute. 
	The faithfulness of the action $\Stiefel_k(\hat\alpha,h)$ is equivalent to the injectivity of $\rho_{\Stiefel}$ by the 5-lemma.
	Pick $c\in C$ such that $\sigma_{\Stiefel}(c)$ is the identity. 
	Then the fibres are fixed and for any $e\in L$ with $h(e,e)=1$, we have $\sigma(c)(e)=e$. So by the $\C$-linearity of $\sigma_{p^{-1}(x)}$, it has to fix $p^{-1}(x)$ pointwise. In other words, $\sigma(c)$ is the identity, 
	hence the injectivity of $\sigma$ forces $c=1$, 
	i.e. $\sigma_{\Stiefel}$ is indeed injective.

	The Grassmann case is similar.  
	Let $(\sigma_1,\rho_1,\tau)\leteq \hat\alpha_1\leteq \hat\alpha\oplus \actionOnTrivialLineBundle{\epsilon}{\tau}$, and action on $p_1\leteq p\oplus\trivialLineBundle{X}$. 
	Suppose that for some $c\in C$, 
	the map $\sigma_1(c)$ fixes all $k$-dimensional spaces of every fibre $p_1^{-1}(x)$.
	Then it necessarily fixes all $1$-dimensional subspaces as well, hence the linear transformation in $p_1^{-1}(x)$ is multiplication by a suitable scalar $z_x\in\C^\times$. 
	But the action $\actionOnTrivialLineBundle{\epsilon}{\tau}$ is trivial on the $1$-dimensional subspace corresponding to the component $\trivialLineBundle{X}$ of $p_1$, so $z_x=1$. 
	Thus $c=1$ is this case by injectivity of $\sigma_1$ (which follows from the injectivity of $\sigma$).
	As above, we can pick a $\rho$-invariant  Hermitian metric $h_1$ on $p_1$, and define the action
	\begin{equation}\label{diag:GrassmannCompactification}
		\begin{tikzcd}[label,column sep=3mm]
			\epsilon \ar[:]\ar[d,"{\Grassmann_k(\hat\alpha\oplus\actionOnTrivialLineBundle{\epsilon}{\tau})}"] & 
			1\ar[r] & 
			C \ar[r,"\iota"] \ar[d,"\sigma_{\Grassmann}"]& 
			G \ar[r,"\pi"] \ar[d,"\rho_{\Grassmann}"] & 
			M \ar[r] \ar[d,"\tau"] &
			1
			\\
			\AutSES{p_{\Grassmann}} \ar[:] & 
			1\ar[r] & 
			\Diff_{p_{\Grassmann}}^\id(\Grassmann_k(L)) \ar[r,inclusion] & 
			\Diff_{p_{\Grassmann}}(\Grassmann_k(L)) \ar[r,"{p_{\Grassmann}}_*"{yshift=2pt}] & 
			\Diff_{p_{\Grassmann}}(X) \ar[r] &
			1
		\end{tikzcd}
	\end{equation}
	where $p_{\Grassmann}\leteq \Grassmann_k(p\oplus\trivialLineBundle{X})$ and 
	\begin{align*}
		\sigma_{\Grassmann}&\from c\mapsto (\generate{e_1,\dots,e_k}\mapsto \generate{\sigma(c)(e_1), \dots,\sigma(c)(e_k)})
		\\
		\rho_{\Grassmann}&\from g\mapsto (\generate{e_1,\dots,e_k}\mapsto \generate{\rho(g)(e_1), \dots,\rho(g)(e_k)}).
	\end{align*}
	These are well defined as above together with the fact that $\sigma$ and $\rho$ act linearly on the fibres. 
	The discussion above shows the injectivity of $\sigma_{\Grassmann}$, hence the faithfulness of this action.
\end{proof}

	\subsection{Proof of \autoref{thm:mainDiff}}
\label{sec:proofB}
\begin{summary}
	In this section, we prove the other main theorem of the paper, \autoref{thm:mainDiff} (from \autopageref{thm:mainBir}), by putting the pieces of the puzzle together obtained in the previous subsections. 
\end{summary}

\begin{proof}[Proof of \autoref{thm:mainDiff}]
	Let $M_r=\torus{2r\lfloor r/2\rfloor}\times \prod_{i=1}^r Y_i$ be any of the compact manifolds from \autoref{rem:manifoldConstruction}, where the Stiefel or Grassmann manifold $Y_i$ corresponds to the choice of parameters $t_i\geq t_0(r)$ and $1\leq k_i\leq t_i+1$.   
	We show that this $M_r$ satisfies the statement.

	Let $G$ be any group as in the statement. 
	Apply \autoref{thm:embedToH} to get non-degenerate Heisenberg groups $\HH(\mu_i\from A_i\times B_i\to C_i)$ with cyclic centre for $1\leq i\leq n\leteq d(\Center(G))$ and a group embedding \[\begin{tikzcd}
		\delta \from G \ar[r,mono] & \prod_{i=1}^n \HH(\mu_i).
	\end{tikzcd}\] 
	
	Applying \autoref{prop:HermitianFormFromHeisenberg} to each $\Z$-bilinear map $\mu_i$ gives an isotropic sublattice data $\mathfrak{D}(i)$ and an isomorphism 
	$\HeisenbergFunctor(\mu_i)\to \HeisenbergFunctor(\mu_{\mathfrak{D}(i)})$ of short exact sequences, in particular a group isomorphism 
	\[\begin{tikzcd}
		\gamma_{{\mathfrak{D}(i)}}\from \HH(\mu_i)\ar[r,iso] & \HH(\mu_{\mathfrak{D}(i)}).
	\end{tikzcd}\] 
	Note that 
	$d_i\leteq \dim_\R(\complexTorus{\mathfrak{D}(i)}) = d(A_i\times B_i) \leq 2\lfloor d(G)/2\rfloor\leq 2\lfloor r/2\rfloor$ by \autoref{prop:HermitianFormFromHeisenberg} as the rank of the group $G$ is at most $r$ by assumption, cf. \autoref{defn:rankGroup}.

	Now \autoref{lem:HeisenbergActionOnLineBudnles} gives a faithful uniformisable action 
	$\alpha_{\mathfrak{D}(i)}\from \HeisenbergFunctor(\mu_{\mathfrak{D}(i)})\acts \pD{\mathfrak{D}(i)}$
	on the bundle $\pD{\mathfrak{D}(i)}\from \LD{\mathfrak{D}(i)}\to \complexTorus{\mathfrak{D}(i)}$ from \autoref{defn:lineBundleToD}.
	The uniformisation of \autoref{prop:complementActionExists} gives a faithful action 
	$\hat \alpha_{\mathfrak{D}(i)} \from \HeisenbergFunctor(\mu_{\mathfrak{D}(i)})\acts p_i$
	 where $p_i$ is the trivial vector bundle over
	$\complexTorus{\mathfrak{D}(i)}$ of rank $r_i\leq R_3(d_i)$. 
	Since $R_3$ from \autoref{rem:R3} is an increasing function, 
	we have 
	$r_i\leq R_3(d_i) \leq R_3(2\lfloor r/2\rfloor)=t_0(r)\leq t_i$ by assumption. 
	Let $\Theta_i$ be the action of $\HeisenbergFunctor(\mu_{\mathfrak{D}(i)})$ on the trivial line bundle $\theta_i$ of rank $1$ over $\complexTorus{{\mathfrak{D}(i)}}$ from \autoref{defn:VvectorBundleActions}. 
	Define the faithful action $\hat\alpha_i\leteq \hat \alpha_{\mathfrak{D}(i)} \oplus \Theta_i^{\oplus(t_i-r_i)}$ 
	on the trivial vector bundle 
	$\hat p_i\leteq p_i\oplus \theta_i^{\oplus(t_i-r_i)}$ 
	of rank $t_i$ with base space $\complexTorus{{\mathfrak{D}(i)}}$.
	To compactify the resulting space, we apply \autoref{lem:compactification} to $\hat\alpha_i$ get faithful actions
	\begin{align*}
		\Stiefel_{k_i}(\hat \alpha_i,h_i)\from \HeisenbergFunctor(\mu_{\mathfrak{D}(i)})&\acts \Stiefel_k(\hat p_i,h_i), 
		\\
		\Grassmann_{k_i}(\hat \alpha_i\oplus\Theta_i)\from \HeisenbergFunctor(\mu_{\mathfrak{D}(i)})&\acts \Grassmann_k(\hat p_i\oplus\theta_i). 
	\end{align*}
	Since $\hat p_i$ is the trivial vector bundle, so is any of the associated compact fibre bundles above.
	Thus the total space of some of these fibre bundles is exactly  $\complexTorus{\mathfrak{D}(i)}\times Y_i$ for the Stiefel or Grassmann manifold $Y_i$  we chose in the beginning of the proof. 
	Hence by restricting that action to the middle group of short exact sequence $ \HeisenbergFunctor(\mu_{\mathfrak{D}(i)})$, we obtain a faithful group action 
	\[\begin{tikzcd}
		\rho_i\from \HH(\mu_{\mathfrak{D}(i)})\ar[r,mono]& \Diff(\complexTorus{\mathfrak{D}(i)}\times Y_i).
	\end{tikzcd}\] 
	
	Note that the isomorphism  $\torus{2\lfloor r/2\rfloor-d_i}\times \complexTorus{\mathfrak{D}(i)} \isom \torus{2\lfloor r/2\rfloor}$ of real manifolds
	induces an embedding 
	\[\begin{tikzcd}
		\Diff(\complexTorus{\mathfrak{D}(i)}\times Y_i) \ar[r,mono]
		& \Diff(\torus{2\lfloor r/2\rfloor}\times Y_i).
	\end{tikzcd}\]

	Finally, put together the maps constructed above into the following diagram.
	\[\begin{tikzcd}[displaystyle]
		G\ar[r,mono,"\text{\eqref{diag:embeddingToHeisenberg}}"',"\delta"] \ar[d,mono,dashed,"\exists"]
		& \prod_{i=1}^{\mathclap{n}} \HH(\mu_i) \ar[r,mono,"\prod_i \gamma_{{\mathfrak{D}(i)}}","\text{\eqref{diag:HeisenbergEmbedsTOHermitian}}"']
		& \prod_{i=1}^{\mathclap{n}} \HH(\mu_{\mathfrak{D}(i)}) 
		\ar[d,mono, shorten <= -1em, shorten >= -1em, "\prod_i \rho_i"',"\eqref{diag:HeisenbergActionOnLineBundle}" pos=0, "\eqref{diag:uniformisationDiff}", "\eqref{diag:StiefelCompactification}/\eqref{diag:GrassmannCompactification}" pos=1] 
		\\
		\Diff(M_r)
		& \prod_{i=1}^{\mathclap{n}} \Diff(\torus{2\lfloor r/2\rfloor}\times Y_i) \ar[l,mono]
		& \prod_{i=1}^{\mathclap{n}} \Diff(\complexTorus{\mathfrak{D}(i)}\times Y_i) \ar[l,mono]
	\end{tikzcd}\]
	The composition above gives the faithful action of $G$ on $M_r$ as required.
\end{proof}


\section*{Acknowledgement} 
The author is grateful to Endre Szabó, László Pyber and Áron Szabó for the fruitful discussions and also to the anonymous referees for their useful suggestions.

\bibliographystyle{amsalpha} 

\begin{thebibliography}{CMiRPS21}

\bibitem[AH61]{AtiyahHirzebruch}
M.~F. Atiyah and F.~Hirzebruch, \emph{Vector bundles and homogeneous spaces},
  Differential geometry (Carl~Barnett Allendoerfer, ed.), Proceedings of
  Symposia in Pure Mathematics, no.~3, American Mathematical Society,
  Providence, RI, 1961, MR:0139181. Zbl:0108.17705., pp.~7--38.

\bibitem[Bir21]{birkar2016singularities}
Caucher Birkar, \emph{Singularities of linear systems and boundedness of Fano
  varieties}, Annals of Mathematics \textbf{193} (2021), no.~2, 347--405.

\bibitem[BL04]{ComplexAbelianVarieties}
Christina Birkenhake and Herbert Lange, \emph{Complex Abelian Varieties}, 2
  ed., Sprin\-ger, 2004.

\bibitem[CMiRPS21]{CsikosMundetPyberSzabo}
Balázs Csikós, Ignasi Mundet~i Riera, László Pyber, and Endre Szabó,
  \emph{On the number of stabilizer subgroups in a finite group acting on a
  manifold}, 2021.

\bibitem[CPS14]{CsPSz}
Balázs Csikós, László Pyber, and Endre Szabó, \emph{Diffeomorphism Groups
  of Compact 4-manifolds are not always Jordan}, 11 2014.

\bibitem[CPS22]{nilpotentJordanHomeo}
\bysame, \emph{Finite subgroups of the homeomorphism group of a compact
  topological manifold are almost nilpotent}, 2022.

\bibitem[CS24]{ChenShramov2022automorphisms}
Yifei Chen and Constantin Shramov, \emph{Automorphisms of surfaces over fields
  of positive characteristic}, Geometry \& Topology \textbf{28} (2024), no.~6,
  2747--2791.

\bibitem[Fis20]{FisherSurvey2008-nonARXIV}
David Fisher, \emph{Groups acting on manifolds: around the Zimmer program},
  Group Actions in Ergodic Theory, Geometry, and Topology, University of
  Chicago Press, 2020, pp.~609--684.

\bibitem[Ghy15]{GhysTalk2015}
Étienne Ghys, \emph{My favourite groups}, Talk at Instituto de Matemática
  Pura e Aplicada, Rio de Janeiro (Brasil), 04 2015.

\bibitem[Gul19]{guld2019finiteDnilpotent}
Attila Guld, \emph{Finite subgroups of the birational automorphism group are
  `almost' nilpotent}, 2019.

\bibitem[Gul20]{guld2020finite2nilpotent}
\bysame, \emph{Finite subgroups of the birational automorphism group are
  'almost' nilpotent of class at most two}, 2020.

\bibitem[Har77]{hartshorne}
Robin Hartshorne, \emph{Algebraic Geometry}, Springer-Verlag New York, 1977.

\bibitem[Hat02]{HatcherAT}
Allen Hatcher, \emph{Algebraic topology}, Cambridge University Press, Cambridge
  New York, 2002.

\bibitem[Hat17]{HatcherVB}
\bysame, \emph{Vector Bundles and $K$-theory}, v2.2 ed., preprint, 11 2017.

\bibitem[HL22]{Hardy1922}
G.~H. Hardy and J.~E. Littlewood, \emph{Some problems of 'Partitio Numerorum':
  IV. The singular series in Waring's Problem and the value of the number
  $G(k)$}, Mathematische Zeitschrift \textbf{12} (1922), no.~1, 161--188.

\bibitem[Hu20]{Hu_2020}
Fei Hu, \emph{Jordan property for algebraic groups and automorphism groups of
  projective varieties in arbitrary characteristic}, Indiana University
  Mathematics Journal \textbf{69} (2020), no.~7, 2493--2504.

\bibitem[Hus93]{Husemoller}
Dale Husemöller, \emph{Fibre Bundles}, 3 ed., Graduate Texts in Mathematics,
  Springer, 1993.

\bibitem[Jor77]{Jordan1877}
Camille Jordan, \emph{Mémoire sur les équations différentielles linéaires
  à intégrale algèbriques}, Journal für die reine und angewandte Mathematik
  \textbf{84} (1877), 89--215.

\bibitem[Kar64]{Karoubi1963-1964}
Max Karoubi, \emph{Les isomorphismes de Chern et de Thom-Gysin en K-théorie},
  Séminaire Henri Cartan \textbf{16} (1964), no.~2, 1--16 (fre).

\bibitem[Lee12]{lee2012smoothmanifolds}
John Lee, \emph{Introduction to smooth manifolds}, Springer, New York London,
  2012.

\bibitem[LP11]{LarsenPink}
Michael Larsen and Richard Pink, \emph{Finite subgroups of algebraic groups},
  Journal of the American Mathematical Society \textbf{24} (2011), no.~4,
  1105--1158.

\bibitem[Mil86a]{Milne1986AbelianVarieties}
J.~S. Milne, \emph{Abelian Varieties}, pp.~103--150, Springer New York, New
  York, NY, 1986.

\bibitem[Mil86b]{Milne1986JacobianVarieties}
\bysame, \emph{Jacobian Varieties}, pp.~167--212, Springer New York, New York,
  NY, 1986.

\bibitem[MiR10]{Riera2010Torus}
Ignasi Mundet~i Riera, \emph{Jordan's theorem for the diffeomorphism group of
  some manifolds}, Proceedings of the American Mathematical Society
  \textbf{138} (2010), 2253--2262.

\bibitem[MiR17a]{MundetT2S2}
\bysame, \emph{Finite groups acting symplectically on $T^2\times S^2$}, Trans.
  Am. Math. Soc \textbf{369} (2017), no.~6, 4457--4483.

\bibitem[MiR17b]{Riera}
\bysame, \emph{Non Jordan groups of diffeomorphisms and actions of compact Lie
  groups on manifolds}, Transformation Groups \textbf{22} (2017), no.~2,
  487--501.

\bibitem[MiR18]{RieraHamSymp}
\bysame, \emph{Finite subgroups of Ham and Symp}, Mathematische Annalen
  \textbf{370} (2018), no.~1, 331--380.

\bibitem[MiR19]{Riera2018}
\bysame, \emph{Finite group actions on homology spheres and manifolds with
  nonzero Euler characteristic}, Journal of Topology \textbf{12} (2019), no.~3,
  744--758.

\bibitem[MiR24]{Mundet2023Survey}
\bysame, \emph{Actions of large finite groups on manifolds}, International
  Journal of Mathematics \textbf{35} (2024), no.~09, 2441012.

\bibitem[MiRSC22]{Riera4dim}
Ignasi Mundet~i Riera and Carles Sáez-Calvo, \emph{Which finite groups act
  smoothly on a given 4-manifold?}, Transactions of the American Mathematical
  Society \textbf{375} (2022), no.~02, 1207--1260.

\bibitem[MS63]{Mann1963}
L.~N. Mann and J.~C. Su, \emph{Actions of elementary $p$-groups on manifolds},
  Transactions of the American Mathematical Society \textbf{106} (1963),
  115--126.

\bibitem[Mum66]{mumford1966equations}
David Mumford, \emph{On the equations defining abelian varieties. I},
  Inventiones mathematicae \textbf{1} (1966), no.~4, 287--354.

\bibitem[Mum08]{mumford2008abelian}
\bysame, \emph{Abelian varieties}, Published for the Tata Institute of
  Fundamental Research By Hindustan Book Agency International distribution by
  American Mathematical Society, Mumbai New Delhi, 2008.

\bibitem[MZ18]{MengZhang2018}
Sheng Meng and De-Qi Zhang, \emph{Jordan property for non-linear algebraic
  groups and projective varieties}, American Journal of Mathematics
  \textbf{140} (2018), no.~4, 1133–1145.

\bibitem[Par08]{park2008complex}
E.~Park, \emph{Complex Topological K-Theory}, Cambridge Studies in Advanced
  Mathematics, Cambridge University Press, 2008.

\bibitem[Pop11]{Popov}
Vladimir~L. Popov, \emph{On the {M}akar--{L}imanov, {D}erksen invariants, and
  finite automorphism groups of algebraic varieties}, Affine algebraic
  geometry, CRM Proc. Lecture Notes, vol.~54, Amer. Math. Soc., Providence, RI,
  2011, pp.~289--311. \MR{2768646}

\bibitem[Pop15]{Popov2013}
\bysame, \emph{Finite subgroups of diffeomorphism groups}, Proceedings of the
  Steklov Institute of Mathematics \textbf{289} (2015), 221--226.

\bibitem[Pop18]{popov2018jordan}
\bysame, \emph{The Jordan property for Lie groups and automorphism groups of
  complex spaces}, Mathematical Notes \textbf{103} (2018), no.~5, 811--819.

\bibitem[PS14]{ProkhorovShramov2014}
Yuri Prokhorov and Constantin Shramov, \emph{Jordan property for groups of
  birational selfmaps}, Compositio Mathematica \textbf{150} (2014), no.~12,
  2054--2072.

\bibitem[PS16]{prokhorov2016jordan}
\bysame, \emph{Jordan property for Cremona groups}, American Journal of
  Mathematics \textbf{138} (2016), no.~2, 403--418.

\bibitem[Ser09]{serre2009Cremona}
Jean-Pierre Serre, \emph{A Minkowski-style bound for the orders of the finite
  subgroups of the Cremona group of rank 2 over an arbitrary field}, Moscow
  Mathematical Journal \textbf{9} (2009), no.~1, 183--198.

\bibitem[Sza19]{specialGroups}
Dávid~R. Szabó, \emph{Special $p$-groups acting on compact manifolds}, 2019.

\bibitem[Sza21]{phd}
\bysame, \emph{Jordan Type Problems via Class 2 Nilpotent and Twisted
  Heisenberg Groups}, Ph.D. thesis, Central European University, 2021.

\bibitem[Sza25]{Heisenberg}
\bysame, \emph{Finite class 2 nilpotent and Heisenberg groups}, Journal of
  Group Theory (2025), 1--32.

\bibitem[Zar74]{ZarhinTrick}
Yuri~G. Zarhin, \emph{A remark on endomorphisms of abelian varieties over
  function fields of finite characteristic}, Mathematics of the USSR-Izvestiya
  \textbf{8} (1974), no.~3, 477.

\bibitem[Zar14]{zarhin2014}
\bysame, \emph{Theta Groups and Products of Abelian and Rational Varieties},
  Proceedings of the Edinburgh Mathematical Society \textbf{57} (2014), no.~1,
  299--304.

\bibitem[Zim14]{Zimmermann2014}
Bruno~P. Zimmermann, \emph{Theta Groups and Products of Abelian and Rational
  Varieties}, Archiv der Mathematik \textbf{103} (2014), no.~2, 195--200.

\end{thebibliography}
\providecommand{\bysame}{\leavevmode\hbox to3em{\hrulefill}\thinspace}
\providecommand{\MR}{\relax\ifhmode\unskip\space\fi MR }
\providecommand{\MRhref}[2]{%
  \href{http://www.ams.org/mathscinet-getitem?mr=#1}{#2}
}
\providecommand{\href}[2]{#2}

\end{document}